 \documentclass[12pt]{article}
 \usepackage[cp1251]{inputenc}
 \usepackage{graphicx}
 \usepackage[english]{babel}
 \usepackage{amsmath,amssymb,euscript,amsthm,amsfonts,mathrsfs,amscd}
 \usepackage{color}
 \usepackage{floatflt}
 \usepackage{mathrsfs} 
 \theoremstyle{plain}
 \newtheorem{theorem}{Theorem}
 \newtheorem{lemma}{Lemma}
 \newtheorem{proposition}{Proposition}
 
 \newtheorem{case}{Case}
 \newtheorem{corollary}{Corollary}
 \theoremstyle{definition}
 \newtheorem{definition}{Definition}

 \theoremstyle{remark}
  \newtheorem{remark}{Remark}

\begin{document}

\title{\emph{Stationary coupling method} for renewal process in continuous time
\\
{\normalsize \emph{application to strong bounds for the convergence rate of the distribution of the regenerative process}}}


\author{G.~Zverkina\thanks{The author is supported by the RFBR, project No 17-01-00633 A.}
}

\maketitle

\begin{abstract}
We propose a new modification of the coupling method for renewal process in continuous time. We call this modification \emph{``the stationary coupling method''}, and construct it primarily to obtain the bounds for convergence rate of the distribution of the regenerative processes in the total variation metrics.
At the same time this modification of the coupling method demonstrates an improvement of the classical result of  polynomial convergence rate of the distribution of the regenerative process -- in the case of a heavy tail.
\\
\textbf{keywords} {Renewal process, Regenerative process, Rate of con\-ver\-gen\-ce, Coupling method }
\\
\textbf{subclass} {MSC 60B10 \and MSC 60J25 \and MSC 60K15}
\end{abstract}

\section{Introduction: Coupling method and its mo\-di\-fi\-ca\-tions}
This paper proposes a new modification of the coupling method, which we call the \emph{stationary coupling method}.

Next, we demonstrate how this method may be used to obtain the bounds for the convergence rate of the distribution of the regenerative process to the stationary distribution in the total variation metrics.

Application of the stationary coupling method demonstrates improvement of the classic results about the polynomial convergence rate of the distribution of the regenerative process.

First, we precise a description of a coupling method.

\subsection{Coupling method (see \cite{zvbib6}).} The coupling method invented by W.~Doeblin in \cite{zvbib1} ordinarily is used to obtain the bounds of convergence rate of a Markov process to the stationary regime.

Below we give a detailed description of this method.

Suppose that $(X_t',\, t\ge 0)$ and $(X_t'',\, t\ge 0)$ are two versions of Markov process $ (X_t,\, t\ge 0)$ with different initial states $X_0'=x'$ and $X_0''=x''$, and with the same transition function; the state space of the process $ (X_t,\, t\ge 0)$ is $\mathscr{X}$ with $\sigma$-algebra $\sigma(\mathscr{X})$.

 In what follows we introduce
 $$ \mathscr{P} ^{x'} _t(A) \stackrel{ \rm{ \;def} } {= \! \! \!=} \mathbf{P} \{X_t' \in A \} , \qquad \mathscr{P} ^{x''} _t(A) \stackrel{ \rm{ \;def} } {= \! \! \!=} \mathbf{P} \{X_t'' \in A \} $$
 for $A\in \sigma(\mathscr{X}),$
  and let
$
 \tau \left(x',x'' \right) \stackrel{ \rm{ \; def} } {= \! \! \!= \! \! \!=} \inf \left \{ t>0: \,X_t'=X_t'' \right \}.
$
Then
$$
 \left| \mathscr{P} _t^{x'} ( A)- \mathscr{P} _t^{x''} ( A) \right| \le \mathbf P \left \{ \tau \left(x',x'' \right) >t \right \}
$$
by the coupling inequality.

The random variable $\tau \left(x',x'' \right)$ is called {\it coupling epoch}.

Now suppose that for some positive increasing unbounded function $ \varphi(t)$ we have $\mathbf{E} \, \varphi \left(\tau \left(x',x'' \right) \right)=C \left(x',x'' \right)< \infty$.
Then from Markov inequality we deduce:
\begin{equation}\label{lab1} 
 \begin{array}{l}
 \left| \mathscr{P} _t^{x'} ( A)- \mathscr{P} _t^{x''} ( A) \right| \le \mathbf P \left \{ \tau \left(x',x'' \right) >t \right \} =
 \\ \\
 \hspace{2.5cm}=\mathbf P \left \{ \varphi \left(\tau \left(x',x'' \right) \right) > \varphi(t) \right \} \le \displaystyle  \displaystyle\frac{\mathbf{E} \, \varphi (\tau (x',x'' ))}{\varphi(t)}.
\end{array}
\end{equation}

Suppose that the process $ (X_t,\, t\ge 0)$ is ergodic, that is, for all initial states $x\in \mathscr{X}$ the distribution $\mathscr{P}_t^x$ converges weakly to the invariant probability measure $\mathscr{P}$ as $t\to \infty$, i.e. $\mathscr{P}_t^x\Longrightarrow \mathscr{P}$ as $t\to \infty$.

Integrating of the inequality (\ref{lab1}) with respect to the stationary measure $ \mathscr{P} $ we obtain
 \begin{equation} \label{zvlab2} 
 \left| \mathscr{P} _t^{x'} ( A)- \mathscr{P} (A) \right| \le  \displaystyle \frac{\displaystyle \int \limits_{ \mathscr{X} } \varphi (\tau (x',x'' ) ) \, \mathrm{d} \mathscr{P} \left (x'' \right)}{\varphi(t)}= \displaystyle\frac{\mathscr{ C} \left (x' \right )}{\varphi(t)} ,
 \end{equation}
and
$$
 \left \| \mathscr{P} _t^{x'} - \mathscr{P} \right \|_{TV} \le 2\,\displaystyle \frac{\mathscr{ C} \left(x' \right)}{\varphi(t)} .
$$

\subsubsection{Successful coupling (see \cite{Grif}) and strong successful coupling.}

The original coupling method was  most commonly used for the Markov chains, i.e. for random processes in discrete time.

It is required to modify application of the coupling method
for random processes in continuous time, since this case suggests $\mathbf P \left \{ \tau \left(x',x'' \right)<+\infty \right \}<1$.

To resolve this problem it was proposed  to construct (in a special probability space) the paired stochastic process $ \left(\mathscr{Z} _t,\, t\ge 0\right)= \left( \left(Z_t',Z_t''\right ),\, t\ge 0\right)$ such that:
\\

{\it 1. $ \;X_t' \stackrel{ \mathscr{D} } {=} Z_t'$ and $X_t'' \stackrel{ \mathscr{D} } {=} Z_t''$ for all $t \ge 0 $;
\\

 2. $\mathbf{P}\left\{\tau (Z_0',Z_0'' )<\infty\right\}=1$, where

~~~$
 \tau \left(Z_0',Z_0'' \right)={ \tau} (\mathscr Z_0 ) \stackrel{ \rm{def} } {= \! \! \!= \! \! \!=} \inf \left \{ t \ge 0: \,Z_t'=Z_t'' \right \};
$
\\

 3. $Z_t'=Z_t''$ for all $t \ge { \tau} \left(Z_0',Z_0'' \right)$.}
 \\

The paired stochastic process $ (\mathscr{Z} _t,\, t\ge 0)= \left(\left (Z_t',Z_t'' \right),t \ge0 \right)$ which satisfies conditions \emph{1}--\emph{3} is called {\it successful coupling} -- see \cite{Grif}.

Let us replace condition \emph{2} by the condition
\\

{\it 2\,$'$. $ \mathbf{E} \, \tau (Z_0',Z_0'' )< \infty$, where
$
 \tau \left(Z_0',Z_0'' \right)={ \tau} (\mathscr Z_0 ) \stackrel{ \rm{def} } {= \! \! \!= \! \! \!=} \inf \left \{ t \ge 0: \,Z_t'=Z_t'' \right \}.
$}
\\

We call the paired stochastic process $ \mathscr{Z} _t= \left(\left (Z_t',Z_t'' \right),t \ge0 \right)$  which satisfies conditions {\it 1}, {\it 2\,}$'$ and {\it 3} {\it the strong successful coupling}.
\vspace{.2cm}

{ \it Note that the processes $(Z_t',\, t\ge 0)$ and $(Z_t'',\, t\ge 0)$ can be non-Markov, and its finite-dimensional distributions may differ from the finite-dimensional distributions of $(X_t',\, t\ge 0)$ and $(X_t'',\, t\ge 0)$ respectively; furthermore, generally speaking, the processes $(Z_t',\, t\ge 0)$ and $(Z_t',\, t\ge 0)$  turn out to be dependent}.
\vspace{.2cm}

Then for all $A \in \sigma (\mathscr{X})$ we  use the coupling inequality in the following form:
 \begin{equation} \label{zvlab3} 
 \begin{array} {l}
 \left| \mathscr{P} ^{x'} _t(A) - \mathscr{P} _t^{x''} (A) \right|= | \mathbf{P} \{ X_t' \in A \} - \mathbf{P} \{ X_t'' \in A \} |= \hspace{3cm}
 \\ \\
 \hspace{1.5cm} = | \mathbf{P} \{ Z_t' \in A \} - \mathbf{P} \{ Z_t'' \in A \} | \le
 \mathbf{P} \{ { \tau} (Z_0',Z_0'' ) \ge t \}
 \le \\ \\
 \hspace{5.5cm}
 \le \displaystyle\frac{ \mathbf {E} \, \varphi ({ \tau} (Z_0',Z_0'' ) )}{\varphi(t)} \le  \displaystyle  \frac{ C (Z_0',Z_0'' )}{\varphi(t)},
 \end{array}
 \end{equation}
where $C (Z_0',Z_0'' )\ge \mathbf{E}\,\varphi(\tau(Z_0',Z_0''))$.

As $Z_0' =X_0'=x' $ and $Z_0'' =X_0''=x'' $, the right-hand side of the inequality depends only on $x' $ and $x'' $; $C (Z_0',Z_0'' )=C (x',x'' )$.
Hence we can integrate the inequality (\ref{zvlab3}) with respect to the stationary measure $ \mathscr{P} $ as in (\ref{zvlab2}):
 $$
 \left| \mathscr{P} ^{x'} _t(A) - \mathscr{P} (A) \right| \le \displaystyle \frac{\displaystyle  \int \limits_{ \mathscr{X} } C \left(x',x'' \right) \mathscr{P} \left(\, \mathrm{d} x'' \right)}{\varphi(t)}  =  \frac{\mathscr{C} \left( x' \right)}{\varphi(t)} ,
$$
and therefore
$$
\left\|\mathscr{P} ^{x'} _t - \mathscr{P} \right\|_{TV}\le 2\, \displaystyle  (\varphi(t))^{-1}\mathscr{C} \left( x' \right).
$$
\vspace{0.2cm}

{ \it However, this integration leads to certain difficulties} -- see, e.g., \cite{{zvbib13},{zvbib14},{zvbib15},{zvbib16}}.

 \subsection{\emph{Stationary coupling method.}}\label{offer} 
In what follows we construct a strong successful coupling \linebreak$ \Big(\mathscr{Z}_t,\, t\ge 0\Big)=\left(\left(Z_t, \widetilde Z_t\right),\, t\ge 0\right)$ for the process $\Big(X_t,\, t\ge 0 \Big)$ with an initial state $x\in\mathscr{X}$ and its stationary version $ \left(\widetilde X_t,\, t\ge 0\right)$.
After that we obtain the estimate for the random variable
 $$ \widetilde \tau(x)= \widetilde \tau(Z_0) \stackrel {{ \; \rm def} } {= \! \! \!=} \inf \left \{ t>0: \, Z_t= \widetilde Z_t \right \}.
 $$
If we prove the finiteness of $\mathbf{E}\,\varphi(\widetilde{\tau}(x))$ then we obtain
$$ \left \| \mathscr{P} ^{x} _t(A) - \mathscr{P} (A) \right \|_{TV} \le 2 \,\mathbf{P} \left \{ \widetilde \tau(x)>t \right \} \le 2\,(\varphi(t))^{-1} \mathbf{E} \,\varphi(\widetilde \tau(x))
$$
analogously to the inequality (\ref{zvlab3}).
\begin{definition}[Stationary Coupling]~

A successful coupling of the Markov process and its stationary version is called \emph{stationary successful coupling}.

We call the method of construction of stationary successful coupling, and the use of this construction for bounds of the convergence rate of the distribution of a Markov process to the stationary distribution the\emph{ stationary coupling method}.
\end{definition}

Our goal is to describe the construction of the stationary successful coupling for renewal process, and application of this construction in order to obtain the bounds for the convergence of the distribution of the regenerative processes to the stationary distribution.

\subsection{The structure of the article}
This article  is divided into 5 Sections, including the Introduction.

In Section 2 we set up main definitions and some necessary denotations.

Section 3 describes the construction of the stationary successful coupling for the backward renewal process when  \emph{Key Condition} is satisfied.

Section 4 demonstrates application of the stationary coupling method for the bounds of the convergence rate of the backward renewal process.

Section 5 extends the results of Section 4 to the regenerative Markov and regenerative non-Markov processes and discusses the way to use the stationary coupling method for the queueing theory.

\section{Some definitions and denotations}

\subsection{Definitions}

\begin{definition}[Renewal Process]

Let $\{ \zeta_{i}\}_{_{i=0}}^{{}^\infty}$ be a sequence of positive independent random variables, and the random variables $\{\zeta_{i}\}_{_{i=1}}^{{}^\infty}$ are identically distributed; denote by $F(s) \stackrel{{\mathrm{def}}}{=\!\!\!=}   \mathbf{P} \{\zeta_i\le s\}$ for $i\ge 1$, and by $G(s) \stackrel{{\mathrm{def}}}{=\!\!\!=}   \mathbf{P} \{\zeta_0\le s\}$.

Suppose that $ \mathbf {E} \,\zeta_i<\infty$ for all $i\ge 0$, and  $\displaystyle \theta_{n} \stackrel{{\mathrm{def}}}{=\!\!\!=} \sum _{i=0}^{n}\zeta_{i}$ for each $n \ge 0$.
Each $\displaystyle \theta_{n}$ is referred to as the $ n^{\scriptsize\mbox{th}}$ renewal time (or renewal point), the intervals $ [\theta_{n},\theta_{n+1}]$
being called renewal intervals, and $\{\zeta_i \stackrel{{\mathrm{def}}}{=\!\!\!=}  \theta_{i+1}-\theta_{i}\}_{_{i=0}}^{{}^\infty}$ being called renewal periods; $\theta_0=\zeta_0$ is called first renewal point.

Then the random variable $  \displaystyle  R_{t} \stackrel{{\mathrm{def}}}{=\!\!\!=}  \sum\limits  _{n=0}^{\infty }\mathbf {1} (\theta_{n}\le t)=\max \left\{\,n:\,\theta_{n}\le t\,\right\}$
(where $ \mathbf {1}(\cdot) $ is the indicator function) represents the number of jumps that have occurred by time $t$, and  we call the process $(R_t,\,t\ge 0)$ a renewal process.

If $\theta_0= \zeta_0 \neq 0$ then the process $(R_t, \,t \ge 0)$ is called delayed.
\end{definition}
\begin{remark}

The renewal process $(R_t,\,t\ge 0)$ is a counting process, and it is not regenerative.
\end{remark}

\begin{definition}[Backward and Forward Renewal Processes]\label{proback}~\\ 
Let $N_t \stackrel{{\mathrm{def}}}{=\!\!\!=}  \left(t-\max\{\theta_n:\,\theta_n\le t\}\right)$ and $N_t^\ast \stackrel{{\mathrm{def}}}{=\!\!\!=}  \left(\min\{\theta_n:\,\theta_n\le t\}-t\right)$, where
$N_t$ is the backward renewal time of the renewal process $R_t$, and $N_t^\ast$ is the forward renewal time of the renewal process $R_t$.

We call the processes $(N_t,\,t\ge 0)$ and $(N_t^\ast,\,t\ge 0)$ \emph{ backward renewal process} and \emph{ forward renewal process} respectively.

At the same time we call the process $(N_t,\,t\ge 0)$ an \emph{embedded backward renewal process} of the renewal process $(R_t,\,t\ge 0)$.
\end{definition}
\begin{remark}
The processes $(N_t,\,t\ge 0)$ and $(N_t^\ast,\,t\ge 0)$ are  Markov piecewise-linear regenerative processes with the state space $\mathscr{R} \stackrel{{\mathrm{def}}}{=\!\!\!=}  \mathbf{R}_{\ge 0}$ with the Borel $\sigma$-algebra $\sigma(\mathscr{R})$.
\end{remark}

Therefore we construct the stationary successful coupling for the backward renewal process $(N_t,\,t\ge 0)$.

\begin{remark}\label{predraspr} 
It is  a well-known fact that if the distribution $F(s)$ is not lattice then
$$
 \lim\limits _{t\to \infty} \mathbf{P} \{N_t\le s\}= \lim\limits _{t\to \infty} \mathbf{P} \{N_t^\ast\le s\} = \widetilde{F}(s),
$$
where
\begin{equation}\label{fla} 
 \widetilde{F}(s)=( \mathbf {E} \,\zeta_1)^{-1} \displaystyle  \int \limits _0^s(1-F(u)) \, \mathrm{d}\,  u.
\end{equation}
\end{remark}

 Below we shall see that application of the stationary coupling method is possible only if the following \emph{Key Condition} is satisfied.
 \vspace{.2cm}

\noindent{\emph {Key Condition.}}
In what follows we suppose that  the following inequality for the cumulative distribution function of the renewal period of the renewal process (or of the length of the regeneration period of the regenerative process) is true, i.e.
$$ \displaystyle
\int \limits_{ \{s: \, \exists F'(s) \} } F'(s) \, \mathrm{d} \, s>0,$$
 and $\mathbf{E} \, \zeta_i< \infty.$
\vspace{0.2cm}

\begin{remark}\label{nenul} 
\emph {Key Condition} for the renewal process implies:
\begin{itemize}
 \item[$\bullet$] $ \mathbf {E} \zeta_i>0$ for $i\ge 0$;
 \item[$\bullet$] There exists an invariant probability distribution $ \mathscr{P} $ on $(\mathscr{R},\sigma(\mathscr{R}))$ which satisfies (\ref{fla}) such that $ \mathscr{P} _t^r\Longrightarrow \mathscr{P} $, where $ \mathscr{P} _t^r(M) \stackrel{{\mathrm{def}}}{=\!\!\!=}  \mathbf{P} \{N_t\in M|\, N_0=r \}$, $M\in \sigma(\mathscr{R})$.
\end{itemize}
\end{remark}
\begin{remark}
For convenience of the reader we assume that the first renewal time of the process $(R_t,\,t\ge0)$ has the cumulative distribution function \linebreak$F_r(s) \stackrel{{\mathrm{def}}}{=\!\!\!=}   \displaystyle \frac{F(r+s)-F(r)}{1-F(r)}$, where $r\ge 0$ is the initial state of the backward renewal process $(N_t,\, t\ge0)$.

$F_r(s)$ is a cumulative distribution function of the residual time of the renewal period if $r$ is a given elapsed time of this period.
\end{remark}

\subsection{Denotations}

\noindent\emph{Denotation 1.} For nondecreasing function $F(s)$ we introduce

\hspace{3cm}$
F^{-1} (y) \stackrel{ \mathrm{def} } {= \! \! \!=} \inf \{x: \,F(x) \ge y \} .
$
 \vspace{.2cm}

\noindent\emph{Denotation 2.} In what follows $ \mu \stackrel{ \mathrm{def} } {= \! \! \!=} \mathbf{E} \, \zeta_1$ and $ \mu_0 \stackrel{ \mathrm{def} } {= \! \! \!=} \mathbf{E} \, \zeta_0$.
\vspace{.2cm}

\noindent\emph{Denotation 3.}  Here and hereafter we put $ \widetilde{F} (s) \stackrel{ \mathrm{def} } {= \! \! \!=}\mu^{-1} \displaystyle \int \limits_{0} ^s (1-F(u)) \, \mathrm{d} \, u$ and

\hspace{2.4cm} $\widetilde{f}(s) \stackrel{{\mathrm{def}}}{=\!\!\!=}  \widetilde{F}'(s)=\mu^{-1}(1-F(u))$.
 \vspace{.2cm}

\noindent\emph{Denotation 4.} For $r\ge 0$  we denote by $F_r(s) \stackrel {{ \rm \; def} } {= \! \! \!=} \displaystyle  \frac{F(s+r)-F(r)} {1-F(r)} $.
 \vspace{.2cm}

\noindent\emph{Denotation 5.}  $\mathscr{U}$, $\mathscr{U}'$, $\mathscr{U}''$, $ \mathscr{U} _i, \mathscr{U} _i', \mathscr{U} _i'', \mathscr{U} _i'''$ are independent uniformly distributed on $[0,1)$ random variables on some probability space $({ \Omega} , { \mathscr{F} } , { \mathbf{P} })$.
 \vspace{.2cm}
\begin{center}
 \begin{figure}
 \centering
\begin{picture}(330,210)
\put(0,0){\vector(1,0){330}}
\put(0,0){\vector(0,1){205}}
\thicklines
\linethickness{0.3mm}
\qbezier(100,0)(115,155)(130,0)
\qbezier(274,0)(280,290)(286,0)
\linethickness{0.7mm}
\put(0,0){\line(1,0){50}}
\put(51,29){\line(1,0){50}}
\put(285,130){\line(1,0){15}}
\put(301,200){\line(1,0){30}}
\put(126,59){\line(1,0){24}}
\put(151,109){\line(1,0){125}}
\linethickness{0.6mm}
\qbezier(101,29)(108,28)(114,43)
\qbezier(114,43)(121,60)(126,59)
\qbezier(276,110)(278,107)(280,119)
\qbezier(280,119)(283,131)(285,130)
\thinlines
\put(0,51){\circle{2}}
\put(2,51){\circle{2}}
\put(4,51){\circle{2}}
\put(6,51){\circle{2}}
\put(8,51){\circle{2}}
\put(10,51){\circle{2}}
\put(12,51){\circle{2}}
\put(14,51){\circle{2}}
\put(16,51){\circle{2}}
\put(18,51){\circle{2}}
\put(20,51){\circle{2}}
\put(22,51){\circle{2}}
\put(24,51){\circle{2}}
\put(26,51){\circle{2}}
\put(28,51){\circle{2}}
\put(30,51){\circle{2}}
\put(32,51){\circle{2}}
\put(34,51){\circle{2}}
\put(36,51){\circle{2}}
\put(38,51){\circle{2}}
\put(40,51){\circle{2}}
\put(42,51){\circle{2}}
\put(44,51){\circle{2}}
\put(46,51){\circle{2}}
\put(48,51){\circle{2}}
\put(50,51){\circle{2}}
\put(52,44){\circle{2}}
\put(54,44){\circle{2}}
\put(56,44){\circle{2}}
\put(58,44){\circle{2}}
\put(60,44){\circle{2}}
\put(62,44){\circle{2}}
\put(64,44){\circle{2}}
\put(66,44){\circle{2}}
\put(68,44){\circle{2}}
\put(70,44){\circle{2}}
\put(72,44){\circle{2}}
\put(74,44){\circle{2}}
\put(76,44){\circle{2}}
\put(78,44){\circle{2}}
\put(80,44){\circle{2}}
\put(82,44){\circle{2}}
\put(84,44){\circle{2}}
\put(86,44){\circle{2}}
\put(88,44){\circle{2}}
\put(90,44){\circle{2}}
\put(92,44){\circle{2}}
\put(94,44){\circle{2}}
\put(96,44){\circle{2}}
\put(98,44){\circle{2}}
\put(100,44){\circle{2}}
\put(102,44){\circle{2}}
\put(104,43){\circle{2}}
\put(106,43){\circle{2}}
\put(108,42){\circle{2}}
\put(110,42){\circle{2}}
\put(112,41){\circle{2}}
\put(114,40){\circle{2}}
\put(116,39){\circle{2}}
\put(118,39){\circle{2}}
\put(120,38){\circle{2}}
\put(122,37){\circle{2}}
\put(124,37){\circle{2}}
\put(126,36){\circle{2}}
\put(128,36){\circle{2}}
\put(130,36){\circle{2}}
\put(132,36){\circle{2}}
\put(134,36){\circle{2}}
\put(136,36){\circle{2}}
\put(138,36){\circle{2}}
\put(140,36){\circle{2}}
\put(142,36){\circle{2}}
\put(144,36){\circle{2}}
\put(146,36){\circle{2}}
\put(148,36){\circle{2}}
\put(150,36){\circle{2}}
\put(152,23){\circle{2}}
\put(154,23){\circle{2}}
\put(156,23){\circle{2}}
\put(158,23){\circle{2}}
\put(160,23){\circle{2}}
\put(162,23){\circle{2}}
\put(164,23){\circle{2}}
\put(166,23){\circle{2}}
\put(168,23){\circle{2}}
\put(170,23){\circle{2}}
\put(172,23){\circle{2}}
\put(174,23){\circle{2}}
\put(176,23){\circle{2}}
\put(178,23){\circle{2}}
\put(180,23){\circle{2}}
\put(182,23){\circle{2}}
\put(184,23){\circle{2}}
\put(186,23){\circle{2}}
\put(188,23){\circle{2}}
\put(190,23){\circle{2}}
\put(192,23){\circle{2}}
\put(194,23){\circle{2}}
\put(196,23){\circle{2}}
\put(198,23){\circle{2}}
\put(200,23){\circle{2}}
\put(202,23){\circle{2}}
\put(204,23){\circle{2}}
\put(206,23){\circle{2}}
\put(208,23){\circle{2}}
\put(210,23){\circle{2}}
\put(212,23){\circle{2}}
\put(214,23){\circle{2}}
\put(216,23){\circle{2}}
\put(218,23){\circle{2}}
\put(220,23){\circle{2}}
\put(222,23){\circle{2}}
\put(224,23){\circle{2}}
\put(226,23){\circle{2}}
\put(228,23){\circle{2}}
\put(230,23){\circle{2}}
\put(232,23){\circle{2}}
\put(234,23){\circle{2}}
\put(236,23){\circle{2}}
\put(238,23){\circle{2}}
\put(240,23){\circle{2}}
\put(242,23){\circle{2}}
\put(244,23){\circle{2}}
\put(246,23){\circle{2}}
\put(248,23){\circle{2}}
\put(250,23){\circle{2}}
\put(252,23){\circle{2}}
\put(254,23){\circle{2}}
\put(256,23){\circle{2}}
\put(258,23){\circle{2}}
\put(260,23){\circle{2}}
\put(262,23){\circle{2}}
\put(264,23){\circle{2}}
\put(266,23){\circle{2}}
\put(268,23){\circle{2}}
\put(270,23){\circle{2}}
\put(272,23){\circle{2}}
\put(274,23){\circle{2}}
\put(276,23){\circle{2}}
\put(278,22){\circle{2}}
\put(280,21){\circle{2}}
\put(282,19){\circle{2}}
\put(284,18){\circle{2}}
\put(286,18){\circle{2}}
\put(288,18){\circle{2}}
\put(290,18){\circle{2}}
\put(292,18){\circle{2}}
\put(294,18){\circle{2}}
\put(296,18){\circle{2}}
\put(298,18){\circle{2}}
\put(300,18){\circle{2}}
\put(302,0){\circle{2}}
\put(304,0){\circle{2}}
\put(306,0){\circle{2}}
\put(308,0){\circle{2}}
\put(310,0){\circle{2}}
\put(312,0){\circle{2}}
\put(314,0){\circle{2}}
\put(316,0){\circle{2}}
\put(318,0){\circle{2}}
\put(320,0){\circle{2}}
\put(322,0){\circle{2}}
\put(324,0){\circle{2}}
\put(-1,100){\line(1,0){2}}
\put(-1,200){\line(1,0){2}}
\put(3,97){0.5}
\put(3,197){1}
\put(150,122){\vector(1,-1){10}}
\put(30,64){\vector(-1,-1){10}}
\put(102,92){\vector(1,-1){10}}
\put(96,97){$f(s)$}
\put(266,142){\vector(1,-1){10}}
\put(260,147){$f(s)$}
\put(24,69){$\widetilde F(s)$}
\put(144,127){$F(s)$}
\end{picture}
 \caption{Illustration to Proposition \ref{kappa} and Remark \ref{zero}. The mixed type cumulative distribution function  $F(s)$ has the positive density $f(s)$ within two intervals, and three jumps.}\label{illu}
 \end{figure}
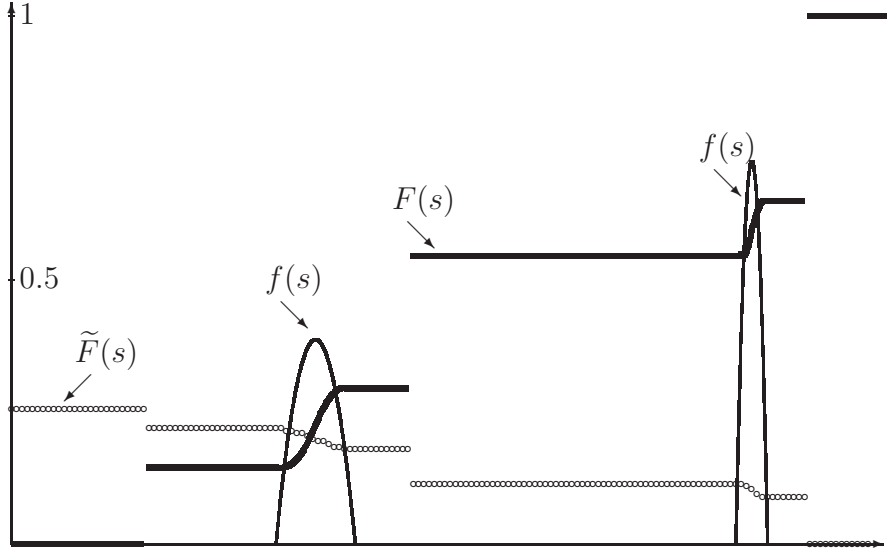
\end{center}

\noindent\emph{Denotation 6.} Denote
$$
\begin{array}{l}
 \varphi(s) \stackrel {{ \rm \; def} } {= \! \! \!=} \mathbf{1} \Big(\exists \, F'(s) \Big) \times \left(F'(s) \wedge \widetilde F'(s) \right)=
 \\ \\
 \hspace{5cm}=\left\{\begin{array}{ll}
                 F'(s) \wedge \widetilde F'(s), & \mbox{if there exists }  F'(s), \\ \\
                 0, & \mbox{otherwise}
               \end{array} \right.
\end{array}
$$
and
$$ \Phi(s) \stackrel {{ \rm \; def} } {= \! \! \!=}  \displaystyle \int \limits_0^s \varphi(u) \, \mathrm{d} \, u;\qquad \kappa \stackrel{{\mathrm{def}}}{=\!\!\!=}  \Phi(+\infty); \qquad \overline{\kappa} \stackrel{{\mathrm{def}}}{=\!\!\!=}  1-\kappa.$$

\begin{proposition}\label{kappa} 
{ Key Condition} implies $ \kappa = \displaystyle  \int\limits_0^ \infty \varphi(s) \, \mathrm{d} \, s>0$.
\end{proposition}
\begin{proof}
Indeed,  \emph{Key Condition} implies that there exists a positive density \linebreak$f(s)=F'(s)$ on some interval $(s_1,s_2)$, $s_1<s_2$.

It is  easy to check that for all $s\in\left (0,s_2\right)$ the inequality $F(s)<1$ is true.
Hence $\widetilde{f}(s)=\widetilde{F}'(s)= \displaystyle \frac{1-F(s)}{\mu} \ge \displaystyle  \frac{1-F(s_2)}{\mu}>0$ for all $s\in \left(s_1,s_2\right)$ and $
\kappa\ge \displaystyle   \int \limits _{s_1}^{s_2}\left(f(s)\wedge \widetilde{f}(s)\right) \, \mathrm{d}\,  s>0
$ -- see Fig. \ref{illu}  for details.

So, Proposition \ref{kappa} is proved. $\bullet$
\end{proof}
\begin{remark}\label{zero} 
If the distribution $F(s)$  which has an absolutely continuous component is close to a discrete distribution then $\kappa$ is close to zero -- see Fig. \ref{illu}.
\end{remark}
\noindent \emph{Denotation 7.} We introduce $ \Psi(s) \stackrel {{ \rm \;def} } {= \! \! \!=} F(s)- \Phi(s)$, $ \widetilde \Psi(s) \stackrel {{ \rm \; def} } {= \! \! \!=} \widetilde F(s)- \Phi(s)$.
\begin{remark} Note that $ \Psi(+ \infty)= \widetilde \Psi(+ \infty)= 1-\kappa[=\overline{\kappa}]$, and the functions $ \Phi(s)$, $ \Psi(s)$ and $\widetilde \Psi(s)$ are nondecreasing.
\end{remark}

 \begin{remark}
It is easily seen that $ \kappa^{-1} \Phi(s)$ is the cumulative distribution function.
Also if $ \kappa<1$ then $\overline{\kappa}{\,}^{-1} \Psi(s)$ and $\overline{\kappa}{\,}^{-1} \widetilde{ \Psi} (s)$ are the cumulative distribution functions.

If $ \kappa=1$ then $ \Phi(s) \equiv F(s) \equiv \widetilde{F} (s)=1-e^{- \lambda s}$ for $\lambda=\mu^{-1}$ and \linebreak$ \Psi(s) \equiv \widetilde{ \Psi} (s) \equiv 0$.
In this case we put $\overline{\kappa}{\,}^{-1} \Psi(s) \stackrel {{ \rm \;def} } {\equiv \! \! \!\equiv} \overline{\kappa}{\,}^{-1} \widetilde{ \Psi} (s) \stackrel {{ \rm \;def} } {\equiv \! \! \!\equiv} 0$ and $ \Psi^{-1} (u)\stackrel {{ \rm \;def} } {\equiv \! \! \!\equiv} \widetilde{ \Psi} ^{-1} (u)\stackrel {{ \rm \;def} } {\equiv \! \! \!\equiv} 0$.
\end{remark}

\noindent\emph{Denotation 8.} Here and hereafter let us introduce
 $$\begin{array}{l}
 \Xi(\mathscr{U} , \mathscr{U} ', \mathscr{U} '') \stackrel {{ \rm \; def} } {= \! \! \!=} \mathbf{1} (\mathscr{U} < \kappa) \Phi^{-1} (\kappa\, \mathscr{U} ')+ \mathbf{1} (\mathscr{U} \ge \kappa) \Psi^{-1} (\overline{\kappa}{\,} \mathscr{U} '');
 \\ \\
 \widetilde \Xi(\mathscr{U} , \mathscr{U} ', \mathscr{U} '') \stackrel {{ \mathrm{def}} } {= \! \! \!=} \mathbf{1} (\mathscr{U} < \kappa) \Phi^{-1} (\kappa\, \mathscr{U} ')+ \mathbf{1} (\mathscr{U} \ge \kappa) \widetilde \Psi^{-1} (\overline{\kappa}{\,} \mathscr{U} '').
 \end{array}
$$

 \begin{remark}\label{kley} 
Clearly,
$$
F(s)= \kappa \left(\kappa^{-1} \Phi(s) \right)+ \overline{\kappa}{\,} \left(\overline{\kappa}{\,}^{-1} \Psi(s) \right)=\Phi(s)+\Psi(s),
$$
 and
$$ \widetilde{F} (s)= \kappa \left(\kappa^{-1} \Phi(s) \right) +\overline{\kappa}{\,} \left(\overline{\kappa}{\,}^{-1} \widetilde{ \Psi} (s) \right)=\Phi(s)+\widetilde{\Psi}(s).$$
Hence,
$$
\begin{array}{l}
 \mathbf{P} \{ \Xi(\mathscr{U} , \mathscr{U} ', \mathscr{U} '') \le s \} = \mathbf{P} \{ \Xi(\mathscr{U} , \mathscr{U} ', \mathscr{U} '') \le s| \mathscr{U} <\kappa\} \mathbf{P} \{ \mathscr{U} <\kappa\}+
 \\ \\
 \hspace{0.3cm}+ \mathbf{P} \{ \Xi(\mathscr{U} , \mathscr{U} ', \mathscr{U} '') \le s| \mathscr{U} \ge \kappa\} \mathbf{P} \{ \mathscr{U} \ge \kappa\}=
 \\ \\
 \hspace{0.3cm}=\kappa \mathbf{P} \{\Phi^{-1}(\kappa  \mathscr{U} ')\le s\}+(1-\kappa) \mathbf{P} \{\Phi^{-1}((1-\kappa)  \mathscr{U} '')\le s\}=
 \\ \\
 \hspace{0.3cm}=\kappa \mathbf{P} \{  \mathscr{U} '\le \Phi(s)\kappa^{-1}\}+(1-\kappa) \mathbf{P} \{  \mathscr{U} ''\le\Phi( s)(1-\kappa)^{-1}\}=F(s).
\end{array}
$$

Analogously,
$\mathbf{P} \{ \widetilde{ \Xi} (\mathscr{U} , \mathscr{U} ', \mathscr{U} '') \le s \} = \widetilde{F} (s)$.

Moreover,
$
\mathbf{P} \{ \Xi(\mathscr{U} , \mathscr{U} ', \mathscr{U} '')= \widetilde{ \Xi} (\mathscr{U} , \mathscr{U} ', \mathscr{U} '') \} = \mathbf{P} \{\mathscr{U} < \kappa\}= \kappa,
$
since the distribution $\widetilde{\Psi}(s)$ is absolutely continuous, and the measure of common part of distributions $\Psi(s)$ and $\widetilde{\Psi}(s)$ is equal to zero.
\end{remark}

\noindent\emph{Denotation 9.} For ~ the ~ random ~ process ~ $(X_t, \, t\ge 0)$ ~ we ~ denote by\linebreak $ \mathscr{P} _t^x(M) \stackrel{{\mathrm{def}}}{=\!\!\!=}  \mathbf{P} \{X_t\in M| X_0=x\}$. ~
If ~ this ~ process ~ is ergodic then \linebreak$ \mathscr{P} (M) \stackrel{{\mathrm{def}}}{=\!\!\!=}  \lim\limits _{t\to \infty} \mathscr{P} _t^x(M)$.

\section{Stationary successful coupling for the backward renewal process $(N_t,\,t\ge 0)$.}\label{konstr} 

This section considers the renewal process $(R_t,\, t\ge 0)$ and its embedded backward renewal process $(N_t,\, t\ge 0)$; from then on we assume that  \emph{Key Condition} is satisfied.

\subsection{Construction of the independent versions of a non-stationary and stationary backward renewal process.}
	
At the beginning, let us recall that the \emph{independent} versions of the processes \linebreak $\Big(N_t,\, t\ge 0\Big)$ and $ \left(\widetilde N_t,\, t\ge 0\right)$ can be constructed as follows (see, e.g., \cite[Chap.V, Proposition 3.5 and Corollary 3.6]{asm}).

\subsubsection{Construction of the version of the non-stationary backward renewal process $(N_t,\, t\ge 0)$ -- see Fig. \ref{inde1}.}\label{nestac} 

If $N_0=r$ then $ \mathbf{P} \{\zeta_0\le s\}=F_r(s)$, where $\zeta_0$ is the first renewal time of the corresponding renewal process $(R_t,\,t\ge 0)$.

We introduce variables $ \zeta_0 \stackrel{ \rm{ \;def} } {= \! \! \!=} F_r ^{-1} (\mathscr{U} _0)$, and $ \zeta_i \stackrel{ \rm{ \;def} } {= \! \! \!=} F^{-1} (\mathscr{U} _i)$ for $i>0;\;$ $\; \theta_i \stackrel{ \rm{ \;def} } {= \! \! \!=} \sum \limits_{j=0} ^i \zeta_j$; then

$$
\begin{array}{l}
 Z_t \stackrel{ \rm{ \;def} } {= \! \! \!=}  \mathbf{1} (t\ge \theta_0)(t- \max \{ \theta_i: \, \theta_i \le t \})+ \mathbf{1} (t<\theta_0)(r+t)=
 \\ \\
 \hspace{4cm}=  \mathbf{1} (t\ge \theta_0)(t- \theta_{R_t})+ \mathbf{1} (t<\theta_0)(r+t)\stackrel{ \mathscr{D} } {=} N_t.
\end{array}
$$
\begin{center}
 \begin{figure}[h]
 \centering
 \begin{picture}(330,40)
\put(0,20){\vector(1,0){330}}
\thicklines
\qbezier(0,30)(20,35)(60,20)
\qbezier(0,30)(20,33)(60,20)
\qbezier(0,30)(20,34)(60,20)
\qbezier(60,20)(80,39)(100,20)
\qbezier(60,20)(80,41)(100,20)
\qbezier(60,20)(80,40)(100,20)
\qbezier(100,20)(140,39)(180,20)
\qbezier(100,20)(140,41)(180,20)
\qbezier(100,20)(140,40)(180,20)
\qbezier(180,20)(200,39)(220,20)
\qbezier(180,20)(200,41)(220,20)
\qbezier(180,20)(200,40)(220,20)
\qbezier(220,20)(250,39)(280,20)
\qbezier(220,20)(250,41)(280,20)
\qbezier(220,20)(250,40)(280,20)
\qbezier(280,20)(290,39)(300,20)
\qbezier(280,20)(290,41)(300,20)
\qbezier(280,20)(290,40)(300,20)
\qbezier(300,20)(320,32)(330,31)
\qbezier(300,20)(320,30)(330,30)
\qbezier(300,20)(320,31)(330,30)
\thinlines
\put(160,10){\line(0,1){30}}
\qbezier(100,20)(130,9)(160,20)
\qbezier(100,20)(130,11)(160,20)
\qbezier(100,20)(130,10)(160,20)
\put(128,7){\scriptsize$Z_t$}
\put(162,10){\scriptsize$t$}
\put(5,35){\scriptsize$\zeta_0\sim F_a$}
\put(65,35){\scriptsize$\zeta_1\sim F$}
\put(125,35){\scriptsize$\zeta_2\sim F$}
\put(185,35){\scriptsize$\zeta_3\sim F$}
\put(235,35){\scriptsize$\zeta_4\sim F$}
\put(210,10){\scriptsize$\theta_3$}
\put(58,10){\scriptsize$\theta_0$}
\put(98,10){\scriptsize$\theta_1$}
\put(178,10){\scriptsize$\theta_2$}
\put(278,10){\scriptsize$\theta_4$}
\put(298,10){\scriptsize$\theta_5$}
\put(0,10){\scriptsize$0$}
\end{picture}
 \caption{Construction of the version $(Z_t, \, t \ge 0)$ of the process $\left(N_t,\, t\ge 0\right)$.}\label{inde1}
 \end{figure}
\end{center}

\subsubsection{Construction of the version of the stationary backward renewal process $\left(\widetilde{N}_t,\, t\ge 0\right)$ -- see Fig. \ref{inde2}.}\label{stac} 

Remark \ref{predraspr}  provides us with the formula of the distribution of the stationary processes $ \left(\widetilde{N} _t,\, t\ge0\right)$ and $ \left(\widetilde{N}^\ast _t,\, t\ge0\right)$: ~ $ \mathbf{P} \left\{\widetilde{N} _t\le s \right\}= \mathbf{P} \left\{\widetilde{N}^\ast _t\le s \right\}=\widetilde{F}(s)$, and therefore $ \mathbf{P} \left\{\widetilde{N} _0\le s \right\}= \mathbf{P} \left\{\widetilde{N}^\ast _0\le s \right\}=\widetilde{F}(s)$.

So, we put $ \widetilde{ \theta} _0\left[=\widetilde{N}_0^\ast\right]= \widetilde \zeta_0 \stackrel{ \rm{ \;def} } {= \! \! \!=} \widetilde F^{-1} (\mathscr{U} _1')$, and
$ \widetilde{ \zeta} _i \stackrel{ \rm{ \;def} } {= \! \! \!=} F^{-1} (\mathscr{U} _i')\;$ for $\;i>0;$ \linebreak $ \widetilde{ \theta} _i \stackrel{ \rm{ \;def} } {= \! \! \!=} \sum \limits_{j=0} ^i \widetilde{ \zeta} _j;\;$ $\; \widetilde Z_0 \stackrel{ \rm{ \;def} } {= \! \! \!=} F^{-1} _{ \widetilde \theta_0} (\mathscr{U} _1'')$ (see Denotation 4); then we put
$$
\widetilde Z_t \stackrel{ \rm{ \;def} } {= \! \! \!=} \mathbf{1} \left(t< \widetilde \theta_0 \right) \left(t+ \widetilde Z_0 \right)+ \mathbf{1} \left(t \ge \widetilde \theta_0 \right) \left(t- \max\left \{ \widetilde \theta_n: \, \widetilde \theta_n \le t \right\} \right) \stackrel{ \mathscr{D} } {=} \widetilde N_t.
$$
 \begin{remark} $\mathbf{P} \{ \widetilde Z_0 \le s \} =$
$$
\begin{array}{l}
   =\displaystyle \int \limits_0^ \infty F_{u} (s) \, \mathrm{d} \widetilde{F} (u)= \displaystyle \int \limits_0^ \infty \frac{F(s+u)-F(u)} {1-F(u)}\times \displaystyle \frac{1-F(u)}{\mu} \, \mathrm{d} u=
 \\
 \hspace{1.7cm}=\mu^{-1}\left( \displaystyle  \int \limits_0^\infty (1-F(u)) \, \mathrm{d}\,  u- \displaystyle  \int \limits_0^\infty (( 1-F(s+u)) \, \mathrm{d}\,  u \right)=
 \\
 \hspace{7cm}=\mu^{-1} \displaystyle  \int \limits_0^s (1-F(u)) \, \mathrm{d} u= \widetilde{F} (s).
\end{array}
$$
 \end{remark}
\begin{center}
 \begin{figure}[h]
 \centering
\begin{picture}(330,45)
\put(0,20){\vector(1,0){330}}
\thicklines
{\qbezier(0,30)(20,30)(40,20)}
{\qbezier(0,31)(20,30)(40,20)}
{\qbezier(0,32)(20,30)(40,20)}
{\qbezier(0,29)(20,30)(40,20)}
{\qbezier(0,28)(20,30)(40,20)}
\qbezier(40,20)(60,41)(80,20)
\qbezier(40,20)(60,40)(80,20)
\qbezier(40,20)(60,39)(80,20)
\qbezier(80,20)(120,41)(160,20)
\qbezier(80,20)(120,39)(160,20)
\qbezier(80,20)(120,40)(160,20)
\qbezier(160,20)(200,41)(240,20)
\qbezier(160,20)(200,39)(240,20)
\qbezier(160,20)(200,40)(240,20)
\qbezier(240,20)(260,41)(280,20)
\qbezier(240,20)(260,39)(280,20)
\qbezier(240,20)(260,40)(280,20)
\qbezier(280,20)(310,31)(330,31)
\qbezier(280,20)(310,31)(330,30)
\qbezier(280,20)(310,30)(330,30)
\thinlines
\put(5,43){\vector(-1,-2){5}}
\put(5,45){\scriptsize$\mathscr{P}$}
\put(10,35){\scriptsize$\widetilde \zeta_0\sim\widetilde F$}
\put(45,35){\scriptsize$\zeta_1\sim F$}
\put(105,35){\scriptsize$\zeta_2\sim F$}
\put(185,35){\scriptsize$\zeta_3\sim F$}
\put(255,35){\scriptsize$\zeta_4\sim F$}
\put(238,10){\scriptsize$\widetilde \theta_3$}
\put(38,10){\scriptsize$\widetilde \theta_0$}
\put(77,10){\scriptsize$\widetilde \theta_1$}
\put(158,10){\scriptsize$\widetilde \theta_2$}
\put(278,10){\scriptsize$\widetilde \theta_4$}
\put(0,10){\scriptsize$0$}
\put(140,10){\line(0,1){30}}
\qbezier(80,20)(110,9)(140,20)
\qbezier(80,20)(110,11)(140,20)
\qbezier(80,20)(110,10)(140,20)
\put(103,5){\scriptsize$\widetilde{Z}_t$}
\put(142,10){\scriptsize$t$}\end{picture}
 \caption{Construction of the version $\left(\widetilde{Z}_t,\, t\ge 0\right)$ of the process $\left(\widetilde{N}_t,\, t\ge 0\right)$.}\label{inde2}
 \end{figure}
\end{center}

\begin{remark}
The processes $\Big(Z_t,\,t\ge 0\Big)$ and $\left(\widetilde {Z}_t,\,t\ge 0\right)$ described in Sections \ref{nestac} and \ref{stac} are {\it independent} since they are constructed by using in\-de\-pen\-dent random variables (see Denotation 5).
\end{remark}

\subsection{Construction of the stationary successful coupling for the backward renewal process.}

Now we construct the successful coupling for the process $\Big(N_t,\, t\ge 0\Big)$ with the initial state $N_0=r$ and and its stationary version $ \left(\widetilde N_t,\, t\ge 0\right)$ with the initial distribution $ \mathbf{P} \left\{\widetilde N_t\le s\right\}=\widetilde{F}(s)$ (on some probability space $ \left({ \Omega} , { \mathscr{F} } , { \mathbf{P} }\right)$ -- see Denotation 5).

Here we use the principles of construction exposed in Sections \ref{nestac} and \ref{stac}.
However, the construction considered in these Sections is the con\-s\-t\-ruc\-tion of the \emph{independent versions} of the processes $\Big(N_t,\, t\ge0\Big)$ and $\left(\widetilde{N}_t,\, t\ge0\right)$.
We have mentioned before that the successful coupling for the processes in continuous time is a pair of \emph{dependent} processes.
So, we  have to modify the construction of  Sections \ref{nestac} and \ref{stac}.

To construct the successful coupling for the processes $\Big(N_t,\, t\ge 0\Big)$ and $ \left(\widetilde N_t,\, t\ge 0\right)$, i.e. a pair of the \emph{dependent} backward renewal processes, it suffices to construct all renewal times of the corresponding renewal processes \linebreak$\Big(R_t,\, t\ge0\Big)$ and $\left(\widetilde{R}_t,\, t\ge0\right)$ -- the times $ \vartheta_i$ on Fig. \ref{suc}.
\begin{remark}
Since the processes $\Big(N_t,\, t\ge 0\Big)$ and $ \left(\widetilde N_t,\, t\ge 0\right)$ are  piecewise-linear processes, the coincidence of these processes can occur only at the common renewal time.
\end{remark}

We construct a pair $ \Big(\mathscr Z_t,\, t\ge0\Big)=\left(\left(Z_t, \widetilde{Z} _t\right),\, t\ge0\right)$ by induction -- see Fig. \ref{suc}.
Since we assume that studied backward renewal process is a homogeneous Markov process, the distribution of the first renewal time of non-stationary version of this process has the distribution which depends only on the initial state. Namely, $G(s)=F_r(s)$ if $N_0=r$.

\subsubsection{Construction of the process $ \Big(\mathscr Z_t,\, t\ge0\Big)$.}\label{constr} 

\noindent\emph{Basis of induction.} We put
$
 \theta_0 \stackrel {{ \rm \; def} } {= \! \! \!=} G ^{-1} (\mathscr{U} _{0})[=F_r^{-1}(\mathscr{U} _{0})]\; $, $\;\widetilde \theta_0 \stackrel {{ \rm \; def} } {= \! \! \!=} \widetilde F^{-1} (\mathscr{U} _{0} ')\;$, \linebreak$\widetilde Z_0 \stackrel {{ \rm \; def} } {= \! \! \!=} F_{ \widetilde \theta_0} ^{-1} (\mathscr{U} _{0} '');
$
here and hereafter $\theta_0$ is the first renewal time of the process $\Big(Z_t, \, t\ge0\Big)$, and $\widetilde \theta_0$ is the first renewal time of the process $\left(\widetilde{Z}_t, \, t\ge0\right)$,  $\widetilde{Z}_0$ has an initial distribution $ \mathscr{P} $ of the stationary backward renewal process, i.e. \linebreak $ \mathbf{P} \left\{\widetilde{Z}_0\le s\right\}=\widetilde{F}(s)$.

Now we introduce $Z_t \stackrel {{ \rm \; def} } {= \! \! \!=} t+r\big[=t+N_0\big]$ and $ \widetilde Z_t \stackrel {{ \rm \; def} } {= \! \! \!=} t+ \widetilde Z_0$ for $t \in[0, \vartheta_0)$, where $ \vartheta_0 \stackrel {{ \rm \; def} } {= \! \! \!=} t_0 \wedge \widetilde t_0$ (on Fig. \ref{suc}: $ \vartheta_0= \widetilde{ \theta} _0$).
Time $\vartheta_0$ is the first time when  a renewal of at least one of the processes $\Big(Z_t, \, t\ge0\Big)$ and $\left(\widetilde{Z}_t, \, t\ge0\right)$ occured.

\begin{center}
 \begin{figure}[h]
 \centering
 \begin{picture}(330,100)
\thinlines
\put(0,70){\vector(1,0){330}}
\put(0,20){\vector(1,0){330}}
\put(11,95){\vector(-1,-1){10}}
\thicklines
\multiput(40,10)(0,5){16}{\line(0,1){3}}
\multiput(150,10)(0,5){16}{\line(0,1){3}}
\multiput(0,10)(0,5){16}{\line(0,1){3}}
\qbezier(0,80)(10,81)(40,70)
\qbezier(0,81)(10,82)(40,70)
\qbezier(0,82)(10,83)(40,70)
\qbezier(0,79)(10,80)(40,70)
\qbezier(0,78)(10,79)(40,70)
\qbezier(0,33)(30,35)(60,20)
\qbezier(0,33)(30,34)(60,20)
\qbezier(0,32)(30,34)(60,20)
\qbezier(40,70)(50,82)(60,80)
\qbezier(40,70)(50,83)(60,81)
\qbezier(40,70)(50,84)(60,82)
\qbezier(40,70)(50,81)(60,79)
\qbezier(40,70)(50,80)(60,78)
\thinlines
\qbezier(40,70)(60,90)(80,70)
\thicklines
\multiput(60,10)(0,5){18}{\line(0,1){3}}
\qbezier(60,20)(80,41)(100,20)
\qbezier(60,20)(80,40)(100,20)
\qbezier(60,20)(80,39)(100,20)
\multiput(100,10)(0,5){18}{\line(0,1){3}}
\qbezier(60,87)(85,88)(100,80)
\qbezier(60,86)(85,87)(100,79)
\qbezier(60,85)(85,86)(100,79)
\qbezier(60,84)(90,85)(100,78)
\qbezier(60,83)(90,84)(100,78)
\thinlines
\qbezier(60,85)(85,89)(120,70)
\thicklines
\qbezier(100,20)(140,39)(180,20)
\qbezier(100,20)(140,40)(180,20)
\qbezier(100,20)(140,41)(180,20)
\multiput(180,10)(0,5){18}{\line(0,1){3}}
\qbezier(100,87)(125,86)(150,70)
\qbezier(100,86)(125,86)(150,70)
\qbezier(100,85)(125,86)(150,70)
\qbezier(100,84)(125,86)(150,70)
\qbezier(100,83)(125,86)(150,70)
\qbezier(150,70)(160,86)(180,80)
\qbezier(150,70)(160,85)(180,79)
\qbezier(150,70)(160,84)(180,79)
\qbezier(150,70)(160,85)(180,78)
\qbezier(150,70)(160,85)(180,78)
\thinlines
\qbezier(150,70)(170,92)(190,70)
\thicklines
\qbezier(180,20)(200,39)(220,20)
\qbezier(180,20)(200,40)(220,20)
\qbezier(180,20)(200,41)(220,20)
\qbezier(180,83)(200,86)(220,70)
\qbezier(180,87)(200,86)(220,70)
\qbezier(180,86)(200,86)(220,70)
\qbezier(180,85)(200,86)(220,70)
\qbezier(180,84)(200,86)(220,70)
\multiput(220,10)(0,5){16}{\line(0,1){3}}
\qbezier(220,20)(250,39)(280,20)
\qbezier(220,20)(250,41)(280,20)
\qbezier(220,20)(250,40)(280,20)
\multiput(280,10)(0,5){16}{\line(0,1){3}}
\qbezier(220,70)(250,91)(280,70)
\qbezier(220,70)(250,89)(280,70)
\qbezier(220,70)(250,90)(280,70)
\qbezier(280,20)(290,39)(300,20)
\qbezier(280,20)(290,41)(300,20)
\qbezier(280,20)(290,40)(300,20)
\qbezier(280,70)(290,91)(300,70)
\qbezier(280,70)(290,89)(300,70)
\qbezier(280,70)(290,90)(300,70)
\qbezier(300,20)(320,32)(330,31)
\qbezier(300,20)(320,31)(330,30)
\qbezier(300,20)(320,31)(330,30)
\multiput(300,10)(0,5){16}{\line(0,1){3}}
\qbezier(300,70)(320,82)(330,81)
\qbezier(300,70)(320,81)(330,80)
\qbezier(300,70)(320,81)(330,80)
\thinlines
\put(9,98){$ \mathscr{P} $}
\put(10,85){\scriptsize${\widetilde\zeta_0\sim \widetilde F}$}
\put(50,96){\vector(1,-4){3}}
\put(35,98){\scriptsize$\widetilde\zeta_1\sim F$}
\put(71,100){\vector(-1,-1){10}}
\put(70,102){$ \mathscr{P} $}
\put(71,90){\scriptsize${\widetilde\Xi_1\sim \widetilde F}$}
\put(111,100){\vector(-1,-1){10}}
\put(110,102){$ \mathscr{P} $}
\put(113,90){\scriptsize${\widetilde\Xi_2\sim \widetilde F}$}
\put(153,85){\scriptsize$\widetilde\zeta_4\sim F$}
\put(191,100){\vector(-1,-1){10}}
\put(190,102){$ \mathscr{P} $}
\put(190,88){\scriptsize${\widetilde\Xi_5\sim \widetilde F}$}
\put(5,35){\scriptsize$\zeta_0\sim F_r$}
\put(65,35){\scriptsize$\Xi_1\sim F$}
\put(120,35){\scriptsize$\Xi_2\sim F$}
\put(185,35){\scriptsize$\Xi_3\sim F$}
\put(235,35){\scriptsize$\zeta_4\sim F$}
\put(235,85){\scriptsize$\widetilde \zeta_5= \zeta_4$}
\put(210,13){\scriptsize$\theta_3$}
\put(52,13){\scriptsize$\theta_0$}
\put(92,13){\scriptsize$\theta_1$}
\put(172,13){\scriptsize$\theta_2$}
\put(272,13){\scriptsize$\theta_4$}
\put(292,13){\scriptsize$\theta_5$}
\put(1,13){\scriptsize$0$}
\put(1,63){\scriptsize$0$}
\put(32,60){\scriptsize$\widetilde \theta_0$}
\put(77,60){\scriptsize$\widetilde \theta_1$}
\put(117,60){\scriptsize$\widetilde \theta_2$}
\put(142,60){\scriptsize$\widetilde \theta_3$}
\put(187,60){\scriptsize$\widetilde \theta_4$}
\put(222,60){\scriptsize$\widetilde \theta_5$}
\put(272,60){\scriptsize$\widetilde \theta_6$}
\put(292,60){\scriptsize$\widetilde \theta_7$}
\put(40,0){\scriptsize$\vartheta_0$}
\put(60,0){\scriptsize$\vartheta_1$}
\put(100,0){\scriptsize$\vartheta_2$}
\put(150,0){\scriptsize$\vartheta_3$}
\put(180,0){\scriptsize$\vartheta_4$}
\put(210,0){\scriptsize$\vartheta_5=\widetilde{\tau}$}
\put(280,0){\scriptsize$\vartheta_6$}
\put(300,0){\scriptsize$\vartheta_7$}
\end{picture}
 \caption{Construction of the successful coupling $\Big(\mathscr{Z}_t,\,t\ge 0\Big).$ }\label{suc}
 \end{figure}
\end{center}

\noindent \emph{ Step of induction.}
Suppose that we have already constructed the process $ \Big(\mathscr Z_t,\, t\ge0 \Big)$ for $t \in [0, \vartheta_n)$, $ \vartheta_n= \theta_i \wedge \widetilde \theta_j$.
Then there are only three alternatives.
\begin{case} In this case we have
$ \vartheta_n= \theta_i= \widetilde \theta_j$ -- on Fig. \ref{suc} this situation occurs for the first time at the point $ \vartheta_5$, and then at the points $\vartheta_6$, $\vartheta_7$, etc.

In this situation at the time $ \vartheta_n$ the processes coincide, and at the same time they begin a new renewal period with the same distribution.

Then we put
$$Z_{ \vartheta_n} = \widetilde Z_{ \vartheta_n} = 0, \qquad \theta_{i+1} = \widetilde \theta_{j+1} = \vartheta_{n+1} =F^{-1} (\mathscr{U} _{n+1})+ \vartheta_{n} ;
$$
and $Z_{t} = \widetilde Z_{t} \stackrel {{ \rm \; def} } {= \! \! \!=} t- \vartheta_n$ for $t \in [ \vartheta_n, \vartheta_{n+1})$.
Thus after the first coincidence (time $ \widetilde{ \tau} = \vartheta_5$ on Fig. \ref{suc}) the processes $\left({Z}_t, \, t\ge0\right)$ and $\left(\widetilde{Z}_t, \, t\ge0\right)$ have  identical renewal periods, and therefore these processes are identical.
\end{case}

\begin{case} In this case we obtain $ \vartheta_n= \widetilde \theta_j< \theta_i$ (the times $ \widetilde{ \theta} _0$ and $ \widetilde{ \theta} _3$ on Fig. \ref{suc}), i.e. the renewal period of the stationary version of our renewal process ended before the renewal period of the nonstationary version of our renewal process ended. In this case we  construct the processes similarly as in the previous Sections, i.e. we put
$$
 \widetilde Z_{ \vartheta_n} = 0, \; \; Z_{ \vartheta_n} = Z_{ \vartheta_n-0} , \qquad \widetilde \theta_{j+1} \stackrel {{ \rm \; def} } {= \! \! \!=} \widetilde \theta_j+F^{-1} (\mathscr{U} _{n+1});
$$
and
$$ \widetilde Z_{t} \stackrel {{ \rm \; def} } {= \! \! \!=} t- \vartheta_n, \qquad Z_{t} \stackrel {{ \rm \; def} } {= \! \! \!=} t- \vartheta_n+Z_{ \vartheta_n} $$
for $t \in[ \vartheta_n, \vartheta_{n+1})$, where $ \vartheta_{n+1} \stackrel {{ \rm \; def} } {= \! \! \!=} \theta_i \wedge \widetilde \theta_{j+1} $.

In fact, in this situation, we construct the process $\left(\widetilde{Z}_t,\, t\ge0\right)$ according to the scheme of Section \ref{stac}, and we do not change anything in the behavior of the process $\Big(Z_t,\, t\ge0\Big)$.
\end{case}
\begin{case} In this case we result in $ \vartheta_n= \theta_i< \widetilde \theta_j$ (the times $ \theta_0$, $ \theta_1$ and $ \theta_2$ on Fig. \ref{suc}): the renewal period of the nonstationary process ended, while the renewal period of the stationary process has not passed yet. In this situation we attempt to continue the behaviour of the processes in such a way that they may coincide with the positive probability (equal to $\kappa$) at the next renewal time: we put
$$
 \theta_{i+1} \stackrel {{ \rm \; def} } {= \! \! \!=} \theta_i+ \Xi(\mathscr{U} _{n+1} , \mathscr{U} _{n+1} ', \mathscr{U} _{n+1} ''); \qquad \widetilde \theta_j \stackrel {{ \rm \; def} } {= \! \! \!=} \theta_i+ \Xi(\mathscr{U} _{n+1} , \mathscr{U} _{n+1} ', \mathscr{U} _{n+1} ''),
$$
 and
$$
Z_t \stackrel {{ \rm \; def} } {= \! \! \!=} t- \vartheta_n, \qquad \widetilde Z_t \stackrel {{ \rm \; def} } {= \! \! \!=} t- \vartheta_n + F^{-1} _{ \varrho} (\mathscr{U} _{n+1} ''')
$$
for $t \in[ \vartheta_n, \vartheta_{n+1})$, where $ \varrho = \widetilde \Xi(\mathscr{U} _{n+1} , \mathscr{U} _{n+1} ', \mathscr{U} _{n+1} '')$ and $ \vartheta_{n+1} \stackrel {{ \rm \; def} } {= \! \! \!=} \theta_i \wedge \widetilde \theta_{j+1} $.

As $ \mathbf{P} \left\{\Xi(\mathscr{U} _{n+1} , \mathscr{U} _{n+1} ', \mathscr{U} _{n+1} '')=\widetilde \Xi(\mathscr{U} _{n+1} , \mathscr{U} _{n+1} ', \mathscr{U} _{n+1} '')\right\}=\kappa>0$ (see Re\-mark \ref{kley}), in this situation $ \mathbf{P} \left\{\theta_i = \widetilde \theta_{j+1} \right\}=\kappa>0$.
\end{case}

\subsubsection{The process $\Big( \mathscr Z_t,\, t\ge0\Big)=\left(\left(Z_t, \widetilde Z_t\right),\, t\ge0\right)$ is a successful coupling for the processes $\Big(N_t,\, t\ge0\Big)$ and $\left(\widetilde{N}_t,\, t\ge0 \right)$.}

\begin{lemma}\label{sovprasp} 
$Z_t \stackrel { \mathscr{D}}{=} N_t$ and $\widetilde{Z}_t \stackrel { \mathscr{D}}{=} \widetilde{N}_t$ for all $t\ge 0$.
\end{lemma}
\begin{proof}
First, we see that the construction of the non-stationary process \linebreak$\Big(Z_t, \, t\ge0\Big)$ is identical to the construction of Section \ref{nestac}.
So, the process $\Big(Z_t, \, t\ge0\Big)$ is Markov, and $Z_t \stackrel { \mathscr{D}}{=} N_t$ for all $t\ge0$.

Now consider the process $\left(\widetilde{Z}_t,\,t\ge0\right)$ for $t\in [0,\theta_0)$. Its construction is identical to the construction of  Section \ref{stac}.

At the time $\theta_0$, this process restarts from the stationary distribution.
At the time $\theta_0$, we have forgotten the previous history of the process $\left(\widetilde{Z}_t,\,t\ge0\right)$, and it began its motion further.
The construction of the process $\left(\widetilde{Z}_t,\,t\ge0\right)$ after time $\theta_0$ is identical to the construction of Section \ref{stac}.
Therefore the distribution of this process is stationary as long as we do not interfere with the the construction of this process.
It means that $\widetilde{Z}_t \stackrel { \mathscr{D}}{=} \widetilde{N}_t$ between the time $\theta_0$ and $\theta_1$.

Now we consider the next intervals $[\theta_i,\theta_{i+1})$.

In this case the process $\left(\widetilde{Z}_t,\,t\ge0\right)$ restarts from the stationary distribution at times $\theta_i$, and up to the next restart (at the time $\theta_{i+1}$) this process has a stationary distribution.

So, for all $t\ge 0$ we have $\widetilde{Z}_t \stackrel { \mathscr{D}}{=}\widetilde{N}_t$ as $ \mathbf {E} (\theta_{i+1}-\theta_i)>0$ (see Remark \ref{nenul}), and Lemma \ref{sovprasp} is proved. $\bullet$
\end{proof}
\begin{lemma}\label{uspex} 
$ \mathbf{P} \{\widetilde{\tau}(r)<\infty\}=1$, where $\widetilde{\tau}(r) \stackrel{{\mathrm{def}}}{=\!\!\!=}  \inf\left\{t\ge0:\, Z_t=\widetilde{Z}_t\Big|N_0=r\right\}$, $r\in \mathscr{R}$.
\end{lemma}
\begin{proof}
Denote by $\mathscr{S} _n \stackrel{ \rm{ \;def} } {= \! \! \!=} \left\{ Z_{ \theta_{n} } = \widetilde{Z} _{ \theta_{n} } \right\} ,$
$$
 \mathbf{S}_n \stackrel{ \rm{ \;def} } {= \! \! \!=} \left(\mathscr{S} _{n}\cap \left(\bigcap \limits_{1\le i\le n-1} \overline{ \mathscr{S} } _i \right)\right)=\left\{Z_{ \theta_{n} } = \widetilde{Z} _{ \theta_{n} } \; \&\linebreak Z_{ \theta_i} \neq \widetilde{Z} _{ \theta_i}, i< n\right \} .
$$

It  can be easily seen that $\mathbf{S}_n\cap \mathbf{S}_m=\circ\!\!\!/\;$ if $n\neq m$, $ \mathbf{P} (\mathbf{S}_n)=\kappa\, \overline{\kappa}\,^{n-1}$, and $ \mathbf{P} \left(\bigcup\limits_{n=1}^\infty \mathbf{S}_n\right)=1$.

 In accordance with our construction of the pair $\Big( \mathscr{Z} _t,\, t\ge 0\Big)$, we obtain \linebreak$ \mathbf{P} \left\{Z_{ \theta_0} \neq \widetilde{Z} _{ \theta_0} \right\} =1$ since the distribution $ \widetilde{F} (s)$ is absolutely continuous, and
$ \mathbf{P} \{ \widetilde{ \tau} = \theta_{n} \} = \mathbf{P} (\mathbf{S} _{n})= \kappa {\,}\overline{\kappa}{\,}^{n-1}$, $n\ge 1$.

It is significant to mention that  events $\mathbf{S}_n$ not depend on $r$, $n\ge 1$.

Note that  random variables $\zeta_i \stackrel{{\mathrm{def}}}{=\!\!\!=}  \theta_i-\theta_{i-1}$ $(i\ge 1)$ are independent since they have been constructed by using different independent random variables.

Then from the law of total probability we deduce:
$$
\begin{array}{l}
  \mathbf{P} \{\widetilde{\tau}(r)<\infty\}= \displaystyle  \sum\limits_{n=1}^\infty\Big( \mathbf{P} \left\{ \widetilde{\tau}(r)<\infty| \mathbf{S}_n\right\} \mathbf{P} \{\mathbf{S}_n\}\Big)=
 \\ \\
 \hspace{1cm}= \displaystyle  \sum\limits _{n=1}^\infty \mathbf{P} (\mathbf{S}_n) \left( \mathbf{P} \{\zeta_n<\infty|\mathbf{S}_n\} \prod\limits_{1\le i\le n-1}  \mathbf{P} \{\zeta_i<\infty|\mathbf{S}_n\} \right)=
\\ \\
 \hspace{1cm}= \displaystyle  \sum\limits _{n=1}^\infty \mathbf{P} (\mathbf{S}_n) \left(\kappa\,^{-1} \Phi(\infty) \prod\limits_{1\le i\le n-1}(\overline{\kappa}\,^{-1}\Psi(\infty)) \right)=\kappa \displaystyle   \sum\limits _{n=1}^\infty \overline{\kappa}\,^{n-1}=1.
\end{array}
$$

If $\kappa=1$ then $ \mathbf{P} (\mathbf{S}_1)=1$, and $\widetilde{\tau}(r)=\theta_1$; $ \mathbf{P} \{\theta_1<\infty\}=1$ if  {\it Key Condition} is satisfied.

So, Lemma \ref{uspex} is proved. $\bullet$
\end{proof}
\begin{lemma}\label{posle} 
$Z_t=\widetilde{Z}_t$ for all $t\ge \widetilde{\tau}(r)$.
\end{lemma}
\begin{proof}
 The statement of the Lemma \ref{posle} follows from the construction of the Case 1 (Section \ref{constr}) and Lemma \ref{posle} is proved.
\end{proof}

\begin{lemma}\label{strogiy} 
$ \mathbf {E} \,\widetilde{\tau}(r)<\infty$.
\end{lemma}
To prove Lemma \ref{strogiy} we need the following elementary
\begin{proposition} Let $\xi$ be a non-negative random variable, and let $\mathscr{E}_1$ and $\mathscr{E}_2$ be an events such that the random variables $\xi$ and $ \mathbf{1} (\mathscr{E}_2)$ are independent. Then
 \begin{equation} \label{uslo} 
 \begin{array}{l}
 \mathbf{E} \Big(\xi \times \mathbf{1} (\mathscr{E}_1)\Big)\le \mathbf{E} \, \xi;
 \\
 \\
 \mathbf{E} \Big(\xi \times \mathbf{1} (\mathscr{E}_2)\Big)= \mathbf{E} \, \xi\, \mathbf{P} (\mathscr{E}_2).
 \end{array}
 \end{equation}
\end{proposition}

\begin{proof}[of Lemma \ref{strogiy}]
Let us recall, that $ \mathbf{P} (\mathbf{S}_{n})=\kappa\,\overline{\kappa}\,^n$, $\zeta_i \stackrel{{\mathrm{def}}}{=\!\!\!=}  \theta_i-\theta_{i-1}$, and for all $i\neq j$ the random variable $\zeta_i$ does not depend on the random event $\mathscr{S}_j$.

Then from (\ref{uslo}) we have for $\kappa\in(0,1)$:
 \begin{equation} \label{zvlab5} 
 \begin{array} {l}
\mathbf{E}\, \widetilde{ \tau}(r) = \mathbf{E} \, \zeta_0+ \mathbf{E} \Big(\mathbf{1} (\mathbf{S} _1)\zeta_1\Big) +
 \mathbf{E} \Big(\mathbf{1} (\mathbf{S}_2)(\zeta_1+ \zeta_2) \Big ) +
 \\ \\
 \hspace{1.5cm}+\mathbf{E} \Big(\mathbf{1} (\mathbf{S}_3)(\zeta_1+ \zeta_2+ \zeta_3)\Big)+ \ldots + \mathbf{E}\left(\mathbf{1}(\mathbf{S}_n) \displaystyle \sum\limits_{i=1}^n\zeta_i\right)+\ldots=
 \\ \\
 \hspace{0.5cm}= \mathbf {E} \,\zeta_0+ \displaystyle   \sum\limits _{i=1}^\infty \left( \mathbf {E} \left(\zeta_i\times \left( \mathbf{1} (\mathbf{S}_i)+  \displaystyle   \sum\limits _{j=i+1}^\infty  \mathbf{1} (\mathbf{S}_j)\right)\right)\right)\le
 \\ \\
 \hspace{0.5cm}\le \mathbf {E} \,\zeta_0+ \displaystyle   \sum\limits _{i=1}^\infty \left( \mathbf {E} \,\zeta_i\times \left(\prod\limits_{\ell=1}^{i-1} \mathbf{P} \left(\overline{\mathscr{S}} _\ell\right)\times\left(1+ \phantom{\prod\limits_{\ell=1}^{i-1}}\right.\right.\right.
 \\ \\
 \hspace{3.5cm}\left.\left.\left.+ \displaystyle  \sum\limits _{j=i+1}^\infty \left(\prod\limits_{\ell=i+1}^{j-1} \mathbf{P} \left(\overline{\mathscr{S}}_\ell \right)\prod\limits_{\ell=j}^\infty \mathbf{P} \left({\mathscr{S}}_\ell\right) \right)\right)\right)\right)=
 \\ \\
 \hspace{0.5cm} =\mathbf{E} \, \zeta_0+  \displaystyle  \sum\limits _{i=1}^\infty\left( \mathbf{E} \, \zeta_i \times \overline{\kappa}{\,}^{i-1} \left(1+ \kappa \displaystyle  \sum \limits_{j=0} ^ \infty\overline{\kappa}{\,}^j \right)\right)=
 \mathbf{E} \, \zeta_0 +2 \kappa^{-1}\mathbf{E} \, \zeta_1,
 \end{array}
 \end{equation}
 here we put $ \displaystyle \prod\limits_{\ell=k}^{k-1}(\cdot) \stackrel{{\mathrm{def}}}{=\!\!\!=}  1$.

If $\kappa=1$ then $ \mathbf{P} (\mathscr{S}_1)=1$.
Hence $ \mathbf {E} \,\widetilde{ \tau}(r) = \mathbf {E} \,\theta_1<\infty$ and Lemma \ref{strogiy} is proved. $\bullet$
\end{proof}

Thus, Lemmata \ref{sovprasp}--\ref{strogiy} imply the following
\begin{theorem}
{The paired process $\Big( \mathscr Z_t,\, t\ge0\Big)=\left(\left(Z_t, \widetilde Z_t\right),\, t\ge0\right)$ constructed in Section \ref{constr} is a strong successful coupling for the backward renewal process $\Big(N_t,\, t\ge 0\Big)$.}
\end{theorem}
\section{Estimation of the coupling epoch of the sta\-ti\-o\-na\-ry successful coupling of the backward renewal process.}

\begin{lemma}\label{poly} 
Let $\Big(N_t,\, t\ge 0\Big)$ be a backward renewal process which satisfies Key Condition with the initial state $N_0=r$, and $\mu_{0,K} \stackrel{{\mathrm{def}}}{=\!\!\!=}  \mathbf {E} (\zeta_0)^K<\infty$, \linebreak$\mu_{K} \stackrel{{\mathrm{def}}}{=\!\!\!=}  \mathbf {E} (\zeta_1)^K<\infty$ for some $K\ge 1$.

Then for the stationary coupling $\Big( \mathscr Z_t,\, t\ge0\Big)$ constructed by the schema exposed in Section \ref{constr} for all $k\in[1,K]$ we have
$$
  \mathbf {E} (\widetilde{\tau}(r))^k= C(k,r)<\infty,
$$
where $\widetilde{\tau}(r) \stackrel{{\mathrm{def}}}{=\!\!\!=}  \inf\left\{t\ge0:\, Z_t=\widetilde{Z}_t\Big|\,N_0=r\right\}$, and
$$
\begin{array}{l}
 C(k,r)\le \mu_{0,k} \kappa  \displaystyle  \sum \limits_{n=1} ^ \infty \left((n+1)^{k-1} \overline{\kappa}{\,}^{n-1} \right)+
 \\
 \hspace{1.2cm} + \mu_k  \displaystyle  \sum \limits_{n=1} ^ \infty \left(\left(\kappa n(n+2)^{k-1} + \overline{\kappa}\,(n+1)^{k-1} \right) \overline{\kappa}{\,}^{n-1} \right) \stackrel{{\mathrm{def}}}{=\!\!\!=}  \widehat{C}(k,\zeta_0).
\end{array}
$$
\end{lemma}
\begin{proof}
Suppose that $\widetilde{\tau}(r)=\theta_\nu$. 

Then by (\ref{uslo}) for $1 \le i< \nu$ we obtain
\begin{equation}\label{nervo}
  \begin{array}{l}
     EE\Big((\zeta_i)^k \mathbf{1} (\mathbf{S}_\nu)\Big)= \mathbf{P} (\mathbf{S}_\nu) \displaystyle  \int \limits _0^\infty s^k \, \mathrm{d}\,  \left(\overline{\kappa}\,^{-1}\Psi(s)\right)\le \mu_k \overline{\kappa}\,^{\nu-2}\kappa;
     \\
     \\
      \mathbf {E} \Big((\zeta_\nu)^k \mathbf{1} (\mathbf{S}_\nu)\Big)=
 \mathbf{P} (\mathbf{S}_\nu) \displaystyle  \int \limits _0^\infty s^k \, \mathrm{d}\,  \left(\kappa^{-1} \Phi(s)\right)\le \mu_k \overline{\kappa}\,^{\nu};
\\
\\
\mbox{and } \mathbf {E} \Big((\zeta_0)^k \mathbf{1} (\mathbf{S}_\nu)\Big)= \kappa\,\overline{\kappa}\,^{\nu} \mu_{0,k}.
  \end{array}
\end{equation}

Using the inequalities (\ref{nervo}) as well as Jensen's inequality for $k \ge 1$ and $a_i\ge 0$ in the form $ \left(\sum \limits_{i=1} ^n a_i \right)^k \le n^{k-1} \sum \limits_{i=1} ^n a_i^k\;$ we  result in the expression similar to formula (\ref{zvlab5}):
\begin{equation}\label{jen} 
 \begin{array} {l}
 \mathbf{E} (\widetilde{ \tau} (r))^k = \mathbf {E}  \left( \displaystyle  \sum\limits _{n=1}^\infty \left( \mathbf{1} (\mathbf{S}_n)\left(\zeta_0+ \displaystyle   \sum\limits _{i=1}^n\zeta_i\right)^k\right)\right) \le
 \\ \\
 \hspace{1.5cm}\le \mathbf{E} \left( \displaystyle  \sum \limits_{n=1} ^ \infty \left((n+1)^{k-1} \left(\zeta_0^k+ \displaystyle  \sum \limits_{i=1} ^n \zeta_i^k \right) \mathbf{1} (\mathbf{S} _n ) \right) \right) =
 \\ \\
 \hspace{1.5cm}= \displaystyle  \sum\limits _{n=1}^\infty\left((n+1)^{k-1} \left( \mathbf {E} \left((\zeta_0)^k \mathbf{1} (\mathbf{S}_n)\right)+ \right.\right.
 \\
 \hspace{1.6cm}\phantom{ \sum\limits _{1\le i\le n-1}}+ \displaystyle   \sum\limits _{1\le i\le n-1} \mathbf {E} \left((\zeta_i)^k \mathbf{1} (\mathbf{S}_n)\right)+ \mathbf {E} \left((\zeta_n)^k \mathbf{1} (\mathbf{S}_n)\right)\le
 \\ \\
 \hspace{1.5cm}\le  \displaystyle  \sum\limits _{n=1}^\infty (n+1)^{k-1}\left(\kappa \, \overline{\kappa}\,^{n-1}\mu_{0,k}+ (n-1)\mu_k \kappa \, \overline{\kappa}\,^{n-1}+ \mu_k\, \overline{\kappa}\,^{n}\right)=
 \\ \\
 \hspace{1.5cm} = \mu_{0,k}\, \kappa  \displaystyle  \sum \limits_{n=1} ^ \infty \left((n+1)^{k-1} \overline{\kappa}{\,}^{n-1} \right)+
 \\
 \hspace{3.5cm} + \mu_k  \displaystyle  \sum \limits_{n=1} ^ \infty \left(\left(\kappa n(n+2)^{k-1} + \overline{\kappa}\,(n+1)^{k-1} \right) \overline{\kappa}{\,}^{n-1} \right).
 \end{array}
\end{equation}

So, Lemma \ref{poly} is proved. $\bullet$
\end{proof}

From now on we introduce denotations $ \mathbf {E} \,e^{\alpha\zeta_0} \stackrel{{\mathrm{def}}}{=\!\!\!=}  \varepsilon_{0,\alpha}$; ~ $ \mathbf {E} \,e^{\alpha\zeta_1} \stackrel{{\mathrm{def}}}{=\!\!\!=}  \varepsilon_{\alpha}$; ~ $\widetilde{\varepsilon}_\alpha \stackrel{{\mathrm{def}}}{=\!\!\!=}   \displaystyle  \int \limits _0^\infty e^{\alpha s} \, \mathrm{d}\, \Psi(s)$.
\begin{remark}
If \emph{Key Condition} is satisfied then $\widetilde{\varepsilon}_0=(1-\kappa)<1$, and there exists $\beta>0$ such that $\widetilde{\varepsilon}_\beta<1$.
\end{remark}
\begin{lemma}\label{exp} 
Let $\Big(N_t,\, t\ge 0\Big)$ be a backward renewal process which satisfies  Key Condition with the initial state $N_0=r$, and $$ \mathbf {E} \,e^{a\zeta_0}=\varepsilon_{0,a}<\infty,\qquad \mathbf {E} \,e^{a\zeta_1}=\varepsilon_{a}<\infty$$ for some $a> 0$.

Suppose $\widetilde{\varepsilon}_\beta<1$ for some $\beta>0$.

Then for the stationary coupling $\Big( \mathscr Z_t,\, t\ge0\Big)$ constructed by the sche\-ma  exposed in Section \ref{constr}  for all $\gamma\in(0,\beta)$  we have
$$
  \mathbf {E} \,e^{\gamma\, \widetilde{\tau}(r)}= \mathscr{C}(\gamma,r)<\infty,
$$
where $\widetilde{\tau}(r) \stackrel{{\mathrm{def}}}{=\!\!\!=}  \inf\left\{t\ge0:\, Z_t=\widetilde{Z}_t\Big|\,N_0=r\right\}$, and
$$
\begin{array}{l}
 \mathscr{C}(\gamma,r)\le \displaystyle \frac{\varepsilon_{0,\beta}\varepsilon_\beta }{1-\widetilde{\varepsilon}_\beta} \stackrel{{\mathrm{def}}}{=\!\!\!=}  \widehat{\mathscr{C}}(\gamma,\zeta_0).
\end{array}
$$
\end{lemma}
\begin{proof}
 Once more assume that $\widetilde{\tau}(r)=\theta_\nu$. Then for $1\le i< \nu$ we obtain
$$
 \mathbf {E} \left(e^{\beta\zeta_i}  \mathbf{1} \left(\overline{\mathscr{S}}_i\right ) \right)= \mathbf{P} \left(\overline{\mathscr{S}}_i\right)  \displaystyle  \int \limits _0^\infty e^{\beta s} \, \mathrm{d}\,  \left(\overline{\kappa}\,^{-1}\Psi(s)\right)= \widetilde{\varepsilon}_\beta
$$
and
$$
 \mathbf {E} \left(e^{\beta\zeta_\nu}  \mathbf{1} (\mathscr{S}_\nu)\right)= \mathbf{P} (\mathscr{S}_\nu) \displaystyle  \int \limits _0^\infty e^{\beta s} \, \mathrm{d}\,  \left(\kappa^{-1} \Phi(s)\right)\le \varepsilon_\beta .
$$
Hence
\begin{equation}\label{ocexp} 
 \begin{array} {l}
 \mathbf{E} (\widetilde{ \tau} (r))^k = \mathbf {E}  \left( \displaystyle  \sum\limits _{n=1}^\infty \left( \mathbf{1} (\mathbf{S}_n)\exp\left(\beta\left(\zeta_0+ \displaystyle  \sum\limits _{i=1}^n\zeta_i\right)\right)\right)\right) =
 \\ \\
\hspace{1.4cm} = \mathbf{E} \left( \displaystyle  \sum \limits_{n=1} ^ \infty \left(e^{\beta\zeta_0}e^{\beta\zeta_n} \mathbf{1} (\mathscr{S} _n )\prod\limits_{1\le i \le n-1}\left(e^{\beta\zeta_i} \mathbf{1} \left(\overline{\mathscr{S}}_i\right) \right) \right) \right) \le
 \\
 \\
 \hspace{6cm}\le  \displaystyle  \sum\limits _{n=1}^\infty \varepsilon_{0,\beta}\varepsilon_\beta \left(\widetilde{\varepsilon}_\beta\right)^{n-1} = \displaystyle \frac{\varepsilon_{0,\beta}\varepsilon_\beta }{1-\widetilde{\varepsilon}_\beta}.
 \end{array}
\end{equation}

Lemma \ref{exp} is proved. $\bullet$
\end{proof}
\begin{corollary}\label{corpoly} 
Let $\Big(N_t,\, t\ge 0\Big)$ be a backward renewal process which satisfies  Key Condition with the initial state $N_0=r$, and $\mu_{0,K} \stackrel{{\mathrm{def}}}{=\!\!\!=}  \mathbf {E} (\zeta_0)^K<\infty$, $\mu_{K} \stackrel{{\mathrm{def}}}{=\!\!\!=}  \mathbf {E} (\zeta_1)^K<\infty$ for some $K\ge 1$.

Then for all $t\ge 0$ and every $k\in[1,K]$ we have
$$
\left\| \mathscr{P} _t^r- \mathscr{P} \right\|_{TV}\le 2 \widehat{C}(k,\zeta_0)t^{-k},
$$
where $ \mathscr{P} _t^r(M) \stackrel{{\mathrm{def}}}{=\!\!\!=}   \mathbf{P} \{N_t\in M|N_0=r\}$, $ \mathscr{P} (M) \stackrel{{\mathrm{def}}}{=\!\!\!=}  \lim\limits _{t\to \infty}  \mathscr{P} _t^r(M)$, $M\in \mathscr{R}$, and $\zeta_0$ has a cumulative distribution function $F_r(s)$.
\end{corollary}

\begin{corollary}\label{corexp} 
Let $\Big(N_t,\, t\ge 0\Big)$ be a backward renewal process which satisfies  Key Condition with the initial state $N_0=r$, and for some $a> 0$
$$ \mathbf {E} \,e^{a\zeta_0}=\varepsilon_{0,a}<\infty,\qquad \mathbf {E} \,e^{a\zeta_1}=\varepsilon_{a}<\infty.
$$

Suppose $\widetilde{\varepsilon}_\beta<1$ for some $\beta>0$.

Then for all $t\ge 0$ and $\gamma\in(0,\beta)$  we derive an estimate
$$
\| \mathscr{P} _t^r- \mathscr{P} \|_{TV}\le 2 \widehat{\mathscr{C}}(\gamma,\zeta_0)e^{-\gamma t},
$$
where $ \mathscr{P} _t^r(M) \stackrel{{\mathrm{def}}}{=\!\!\!=}   \mathbf{P} \{N_t\in M|N_0=r\}$, $ \mathscr{P} (M) \stackrel{{\mathrm{def}}}{=\!\!\!=}  \lim\limits _{t\to \infty}  \mathscr{P} _t^r(M)$, $M\in \mathscr{R}$, and $\zeta_0$ has a cumulative distribution function $F_r(s)$.
\end{corollary}

\begin{proof}
  Corollary \ref{corpoly} and  Corollary \ref{corexp} follow from Lemma \ref{poly}, Lemma \ref{exp} and Section \ref{offer}.
\end{proof}

\begin{remark}
The Corollary \ref{corpoly} improves the classical result: there exists the constant $C$ such that $\left\| \mathscr{P} _t^r- \mathscr{P} \right\|_{TV}\le C\,t^{-K+1}$ (see \cite{asm,bor,zvbib5,thor})
\end{remark}

\section{Application of the stationary coupling to \\ strong bounds for the convergence rate of the di\-s\-tri\-bu\-tion of the regenerative process}

\subsection{Introduction to regenerative processes}
Let us recall the definition of a regenerative process in continuous time.
\begin{definition}
Assume that the random process $\Big(X_t, \, t\ge0\Big)$ with the state space $(\mathscr{X},\sigma(\mathscr{X}))$ has continuous time parameter $t\in[0,+\infty)$.

Besides, we suppose that there exists a renewal process $\Big(R_t, \, t\ge 0\Big)$ with the renewal times $\{\theta_j\}_{{}_{j=0}}^{{}^\infty}$; as before, $\theta_i= \displaystyle  \sum\limits _{j=0}^i\zeta_j$, the random variables $\{\zeta_j\}_{{}_{j=0}}^{{}^\infty}$ are independent, and $\{\zeta_j\}_{{}_{j=1}}^{{}^\infty}$ are identically distributed.

Furthermore, the pair of processes $((X_t,R_t),\, t\ge 0)$ has the following property: for each $n \ge 0$, the post-$\theta_n$ process $\Big(X_{\theta_n + t},\, t \ge 0\Big)$ is independent of $(\theta_0,\ldots,\theta_n)$ (or, equivalently, of $\zeta_0,\ldots,\zeta_n)$ and its distribution
does not depend upon $n$.

Then we call the process $\Big(X_t, \, t\ge0\Big)$ \emph{regenerative} process.

We call $(R_n, \, t\ge0)$ the \emph{embedded renewal
process} and refer to the $\theta_n$ as \emph{regeneration points} or regeneration times.

The behaviour of the process $X_t$ on the $k^{\scriptsize\mbox{th}} $ cycle \linebreak$\Theta_k \stackrel{{\mathrm{def}}}{=\!\!\!=}  (X_{t+\theta_k},\, t\in[0, \zeta_{k+1}])$ is a random element with the state space $D_0(\mathscr{X})$ of $\mathscr{X}$-valued functions which are right-continuous and have left-hand limits and with finite lifelengths -- see \cite[Chapter 6]{asm}; $\{\Theta_k\}_{{}_{k=0}}^{{}^\infty}$ are i.i.d. random elements.

We call the intervals $(\theta_{t-1},\theta_t)$ the \emph{regeneration periods}, and we call the random variable $\zeta_k$ the \emph{length of the $k^{\scriptsize{th}}$ regeneration period}.

When $t_0$ equals $0$, $\Big(X_t, \, t\ge 0\Big)$ is called a nondelayed (or pure) regenerative process.
Otherwise, the process is called a delayed regenerative process.
\end{definition}

\subsection{Bounds for the convergence rate of the regenerative process.}\label{final} 

This Section is devoted to \emph{Markov} regenerative processes.
We attempt to derive the strong bounds (in the total variation metrics) for the convergence rate of this process in the case when the distribution of the  regeneration period length satisfies  \emph{Key Conditions}.

These bounds can be extended to the case of non-Markov regenerative processes by the following reasons.

\begin{remark}[On non-Markov regenerative processes]

Every arbitrary non-Markov regenerative process $\Big(X_t, \, t\ge0\Big)$ with the state space $(\mathscr{X},\sigma(\mathscr{X}))$ can be extended to the Markov regenerative process $\Big(\overline{X}_t, \, t\ge 0\Big)$ with the extended state space $\left( \overline{ \mathscr{X} }, \sigma \left(\overline{ \mathscr{X} } \right) \right) $.

For instance, we can include in the state $X_t$ for $t \in[ \theta_{n-1} , \theta_n)$ full history of the process $\Big(X_t, \, t\ge 0\Big)$ on the time interval $[ \theta_{n-1} ,t]$ for markovization of non-Markov regenerative process.

The process $ \overline X_t \stackrel {{ \rm \; def} } {= \! \! \!=} \left \{ X_s, s \in[ \theta_{n-1} ,t] |\;t \in[ \theta_{n-1} , \theta_n) \right \} $ is Markov and re\-ge\-ne\-ra\-ti\-ve with the extended state space $\left(\overline{ \mathscr{X} },\sigma \left(\overline{ \mathscr{X} } \right)\right) $.

Denote by $ \overline{ \mathscr{P} } _t^{\; \overline{x}} \left(\overline{M}\right) \stackrel{ \rm{ \;def} } {= \! \! \!=} \mathbf{P} \left\{ \overline{X} _t \in \overline{M} \right\} $ for the process $ \overline{X} _t$ with the initial state $ \overline{X} _0=\overline{x}$ and $\overline{M} \in \sigma \left(\overline{ \mathscr{X} }\right)$.

If $ \mathbf{E} \, \zeta_i< \infty$ then $ \overline{ \mathscr{P} } _t^{ \;\overline{x} } \Longrightarrow \overline{ \mathscr{P} } $, where $\overline{ \mathscr{P} }$ is  some stationary probability measure on the state space $\left(\overline{ \mathscr{X} },\sigma \left(\overline{ \mathscr{X} } \right)\right) $.

If we  prove that $ \left \| \overline{ \mathscr{P} } _t^{\; \overline{x}} - \overline{ \mathscr{P} } \right \|_{TV} \le \psi(t, \overline{x})\big[=\phi(t,\zeta_0)\big]$ for all $t \ge 0$ then this inequality will be true for the original non-Markov regenerative process $X_t$:
$$
\begin{array}{l}
 \left\| { \mathscr{P} } _t^x - { \mathscr{P} } \right\|_{TV} \stackrel{{\mathrm{def}}}{=\!\!\!=}  2 \sup\limits _{M\in\sigma(\mathscr{X})}\left| \mathscr{P} _t^x(M)- \mathscr{P} (M)\right| \le
 \\ \\
\hspace{2.1cm}\le\left\| \overline{ \mathscr{P} } _t - \overline{ \mathscr{P} } \right\|_{TV} \stackrel{{\mathrm{def}}}{=\!\!\!=}  2 \sup\limits _{\overline{M}\in\sigma\left (\overline{\mathscr{X}} \right)}\left|\overline{ \mathscr{P} }_t\left (\overline{M}\right) -\overline{ \mathscr{P} }\left(\overline{M}\right)\right| \le \phi(t, \zeta_0).
\end{array}
$$

Besides, the extension for markovization can be more simple for the queueing non-Markov regenerative process (when considering the structure of this process).
\end{remark}

So, we deal with the regenerative Markov process $\Big(X_t,\, t\ge 0\Big)$ with the state space $(\mathscr{X},\sigma(\mathscr{X}))$; and let $\{\theta_i\}_{{}_{i=0}}^{{}^\infty}$ be its regeneration points.
As  was mentioned before,  $\theta_{i}= \sum\limits _{k=0}^i\zeta_k$ and $\zeta_k\ge 0$; $\{\zeta_i\}_{{}_{i=0}}^{{}^\infty}$ are independent non-negative random variables, and $\{\zeta_i\}_{{}_{i=1}}^{{}^\infty}$ are identically distributed.

Let the renewal process $\Big(R_t,\, t\ge 0\Big)$ with the renewal points $\{\theta_i\}_{{}_{i=0}}^{{}^\infty}$ be an embedded renewal process of the regenerative process $\Big(X_t,\, t\ge 0\Big)$, and let $\Big(N_t,\, t\ge 0\Big)$ be an embedded backward renewal process of the renewal process $\Big(R_t,\, t\ge 0\Big)$.

Once more we denote by $F(s) \stackrel{{\mathrm{def}}}{=\!\!\!=}   \mathbf{P} \{\zeta_i\le s\}$ for $i\ge 1$, and by \linebreak$G(s) \stackrel{{\mathrm{def}}}{=\!\!\!=}   \mathbf{P} \{\zeta_0\le s\}$; we assume that $G(s)=F_r(s)$ -- see Denotation 4, and we suppose that  \emph{Key Condition} is satisfied.

A paired process $\Big(V_t,\, t\ge 0\Big) \stackrel{{\mathrm{def}}}{=\!\!\!=} \Big((X_t,N_t),\, t\ge 0\Big)$ is a Markov re\-ge\-ne\-ra\-ti\-ve process with the state space $ \left( \overline{ \mathscr{X}}, \sigma \left( \overline{\mathscr{X}} \right) \right) \stackrel{{\mathrm{def}}}{=\!\!\!=}  (\mathscr{X}, \sigma(\mathscr{X})) \times(\mathscr{R},\sigma(\mathscr{R}))$, and the components $X_t$ and $N_t$ of the process $\Big(V_t,\, t\ge 0\Big)$ are dependent.
Namely, there exists a conditional distribution
$$
G_a(M) \stackrel{{\mathrm{def}}}{=\!\!\!=}   \mathbf{P} \{X_t\in M|\, N_t=a\}= \mathbf{P} \{X_{\theta_k+a}\in M|\,\theta_k+a\le \theta_{k+1}\},
$$
where $M\in \sigma(\mathscr{X})$.

So, if we know the renewal times of the process $\Big(R_t,\, t\ge 0\Big)$ then it is possible to define the (conditional) distribution of the process $\Big(X_t,\, t\ge 0\Big)$ at any time $\Big($given the values of $\{\theta_i\}_{{}_{i=0}}^{{}^\infty}\Big)$, and we can propose the following method of construction of the successful coupling $\left( \left(W_t,\widetilde{W_t} \right),\, t\ge 0\right)$ for the processes  
$$\Big(V_t,\, t\ge 0\Big)=\Big((X_t,N_t),\, t\ge 0\Big)\mbox{ and } \left(\widetilde{V}_t,\, t\ge 0\right) \stackrel{{\mathrm{def}}}{=\!\!\!=} \left( \left(\widetilde{X}_t,\widetilde{N}_t \right),\, t\ge 0\right).$$

First,  it is possible to construct (on the probability space $(\Omega,\mathscr{F}, \mathbf{P} )$ -- see Denotation 5) the stationary successful coupling for the second parts of the processes $\Big(V_t,\, t\ge 0\Big)$ and $\left(\widetilde{V}_t,\, t\ge 0\right)$, i.e. for the processes $\Big(N_t,\, t\ge 0\Big)$ and $\left(\widetilde{N}_t,\, t\ge 0\right)$ by the schema described in Section \ref{konstr}.
As a result, we obtain the process $\Big(\mathscr{Z}_t,\, t\ge 0\Big)=\left( \left(Z_t,\widetilde{Z}_t\right),\, t\ge 0\right)$.

Then we fix the time $\widetilde{\tau}(r) \stackrel{{\mathrm{def}}}{=\!\!\!=}  \inf\left\{t\ge 0 :\, Z_t=\widetilde{Z}_t\Big| \, N_0=r\right\}$; we do know that $ \mathbf {E} \,\widetilde{\tau}(r)<\infty$ (Lemma \ref{strogiy}); and $\widetilde{\tau}(r)=\theta_j=\vartheta_i$ for some $i$ and $j$.

After that we can complete the process $\Big(Z_t,\, t\ge 0\Big)$ to the process \linebreak$\Big((Y_t,Z_t),\, t\ge 0\Big)$ by construction (on some probability space $(\Omega',\mathscr{F}', \mathbf{P} ')$) the random elements
$$\Theta_0 \stackrel{{\mathrm{def}}}{=\!\!\!=} \{Y_t, \, t\in [0,\theta_0]|\, \theta_0=\zeta_0\} \stackrel { \mathscr{D}}{=} \{X_t, \, t\in [0,\theta_0]|\, \theta_0=\zeta_0\}$$
and for $k=1,2,\ldots j$
$$\Theta_k \stackrel{{\mathrm{def}}}{=\!\!\!=} \{Y_t, \, t\in [\theta_{k-1},\theta_k]|\, \theta_k-\theta_{k-1}=\zeta_k\} \stackrel { \mathscr{D}}{=} \{X_t, \, t\in [\theta_{k-1},\theta_k]|\, \theta_k-\theta_{k-1}=\zeta_k\}.$$

Similarly, it is possible to complete the process $\left(\widetilde{Z}_t,\, t\ge 0\right)$ to the process $ \left( \left(\widetilde{Y}_t,\widetilde{Z}_t \right),\, t\ge 0 \right)$ by construction (on some probability space $(\Omega'',\mathscr{F}'', \mathbf{P} '')$) the random elements
$$
\widetilde{\Theta}_0 \stackrel{{\mathrm{def}}}{=\!\!\!=} \left\{\widetilde{Y}_t, \, t\in [0,\widetilde{\theta}_0]\Big|\, \widetilde{\theta}_0=\widetilde{\zeta}_0\right\} \stackrel { \mathscr{D}}{=} \left\{\widetilde{X}_t, \, t\in [0,\widetilde{\theta}_0] \Big|\, \widetilde{\theta}_0= \widetilde{\zeta}_0 \right\}
$$
and for $k=1,2,\ldots i$
$$
\begin{array}{l}
 \widetilde{\Theta}_k \stackrel{{\mathrm{def}}}{=\!\!\!=} \{\widetilde{Y}_t, \, t\in \left[\widetilde{\theta}_{k-1},\widetilde{\theta}_k \right]\Big|\, \widetilde{\theta}_k- \widetilde{\theta}_{k-1}= \widetilde{\zeta}_k\}  \stackrel { \mathscr{D}}{=}
 \\
 \hspace{5cm} \stackrel { \mathscr{D}}{=}\left\{\widetilde{X}_t, \, t\in \left[\widetilde{\theta}_{k-1},\widetilde{\theta}_k \right]\Big |\, \widetilde{\theta}_k-\widetilde{\theta}_{k-1}= \widetilde{\zeta}_k\right\}.
\end{array}
$$

At the time $\widetilde{\tau}(r)=\theta_j=\vartheta_i$ the processes $\left(Z_t,\, t\ge 0\right)$ and $\left(\widetilde{Z}_t,\, t\ge 0\right)$ have the same distribution (or even coincide in particular cases).

After the time $\widetilde{\tau}(r)=\theta_j$ we can again construct (on some probability space $(\Omega''',\mathscr{F}''', \mathbf{P} ''')$) the random elements
$$
{\Theta}_k \stackrel{{\mathrm{def}}}{=\!\!\!=} \{{Y}_t, \, t\in \left[{\theta}_{k-1},{\theta}_k \right]|\, {\theta}_k- {\theta}_{k-1}= {\zeta}_k\}  \stackrel { \mathscr{D}}{=} \{{X}_t, \, t\in \left[{\theta}_{k-1},{\theta}_k \right]|\, \theta_k-\theta_{k-1}=\zeta_k\}
$$
for $k>j$, and we put $\widetilde{\Theta}_{i+\ell} \stackrel{{\mathrm{def}}}{=\!\!\!=}  \Theta_{j+\ell}$, $\ell>0$.

Analogously, it is possible to construct the process
$$
\left( \left(W_t,\widetilde{W}_t \right),\, t\ge 0\right)=\left(\left(Y_t,Z_t,\widetilde{Y}_t,\widetilde{Z}_t\right),\, t\ge 0\right) \stackrel { \mathscr{D}}{=} \left( \left(X_t, N_t,\widetilde{X}_t, \widetilde{N}_t\right), \, t\ge 0\right),
 $$
on the probability space
 $$
\left( \widetilde{\Omega}, \widetilde{\mathscr{F}}, \widetilde{ \mathbf{P} }\right)  \stackrel{{\mathrm{def}}}{=\!\!\!=} (\Omega,\mathscr{F}, \mathbf{P} ) \times (\Omega',\mathscr{F}', \mathbf{P} ') \times (\Omega'',\mathscr{F}'', \mathbf{P} ''') \times (\Omega''',\mathscr{F}''', \mathbf{P} '''),
$$
and this process is a successful coupling for the processes
\linebreak
$\Big(V_t,\, t\ge 0\Big)=\Big((X_t,N_t),\, t\ge 0\Big)$ and $\left(\widetilde{V}_t,\, t\ge 0\right)=\left(\left (\widetilde{X}_t,\widetilde{N}_t \right),\, t\ge 0 \right)$ -- here  and hereafter we  omit a detailed description of the construction of this process.

Now, from  Corollary \ref{corpoly} and Corollary \ref{corexp} we deduce:

\begin{theorem}
Let $\Big(X_t,\, t\ge 0\Big)$ be a Markov regenerative process with the state space $(\mathscr{X},\sigma(\mathscr{X}))$ which satisfies  Key Condition, with the initial state $X_0=x$, and suppose that $\mu_{0,K} \stackrel{{\mathrm{def}}}{=\!\!\!=}  \mathbf {E} (\zeta_0)^K<\infty$, $\mu_{K} \stackrel{{\mathrm{def}}}{=\!\!\!=}  \mathbf {E} (\zeta_1)^K<\infty$ for some $K\ge 1$.

Then for all $t\ge 0$ and every $k\in[1,K]$  we derive an estimate
$$
\left\| \mathscr{P} _t^x- \mathscr{P} \right\|_{TV}\le 2 \, \widehat{C}(k,\zeta_0)t^{-k},
$$
where $ \mathscr{P} _t^r(M) \stackrel{{\mathrm{def}}}{=\!\!\!=}   \mathbf{P} \{X_t\in M|X_0=x\}$, $ \mathscr{P} (M) \stackrel{{\mathrm{def}}}{=\!\!\!=}  \lim\limits _{t\to \infty}  \mathscr{P} _t^x(M)$, $M\in \sigma(\mathscr{X})$, and $\zeta_0$ is a first regeneration point of the process $\Big(X_t,\, t\ge 0\Big)$.
\end{theorem}

\begin{theorem}
Let $\Big(X_t,\, t\ge 0\Big)$ be a Markov regenerative process with the state space $(\mathscr{X},\sigma(\mathscr{X}))$ which satisfies  Key Condition, with the initial state $X_0=x$, and let $ \mathbf {E} \,e^{a\zeta_0}=\varepsilon_{0,a}<\infty$, $ \mathbf {E} \,e^{a\zeta_1}=\varepsilon_{a}<\infty$ for some $a> 0$.

Then for all $t\ge 0$ and for all $\beta>0$ such that $\widetilde{\varepsilon}_\beta<1$  we have
$$
\| \mathscr{P} _t^x- \mathscr{P} \|_{TV}\le 2\, \widehat{\mathscr{C}}(\beta,\zeta_0)e^{-\beta t}.
$$
\end{theorem}
\begin{remark}
Furthermore, it is possible to obtain a certain decrease in value of the constants $\widehat{C}(k,\zeta_0)$ and $\widehat{\mathscr{C}}(\beta,\zeta_0)$  using the properties of the cumulative distribution function $F(s)$ and  determining more  accurately  the estimates in the calculations (\ref{jen}) end (\ref{ocexp}).
\end{remark}

\subsection{ Application to the queueing theory.}
The distribution of the period of the queueing regenerative process $(Q_t,\,t\ge 0)$ is often unknown in the queuing theory.
However, the regeneration period can be often split into two parts, in most cases we call them a busy period and an idle period.
Furthermore, as a rule the idle period has a known non-discrete distribution.
We suppose that the bounds for the moments of the busy period are also known.
This queueing process has an \emph{embedded alternating renewal process}, and it turns out not to be Markov, although the backward alternating renewal process defined similarly to the backward renewal process is Markov -- see Definition \ref{proback}.
So, in this situation the queueing regenerative process has an \emph{embedded backward alternating renewal process} $(A_t,\,t\ge 0)$.

The stationary coupling method can be applied to the backward al\-ter\-na\-ting renewal process $(A_t,\,t\ge 0)$ in the case when one of its alternating renewal periods has the cumulative distribution function   which satisfies  \emph{Key Condition}, and all alternating renewal periods have finite expectation.
Mo\-re\-o\-ver, it is possible to use the stationary coupling method for to the backward alternating renewal process in  some cases when the alternating periods of one regenerative period are dependent.
The description of the stationary coupling method for the backward alternating process will be presented in the next publications.

So, we can find the bounds for the convergence of the backward renewal process $(A_t,\,t\ge 0)$ by using the stationary coupling method.
Then, using the argumentation of Section \ref{final}, we can verify that the bounds for the process $(A_t,\,t\ge 0)$ appear to be true for a complete process $((Q_t,A_t),\, t\ge 0)$ and for the process $(Q_t,\,t\ge 0)$.

Besides, if the bounds of moments of a busy period are also known, we can apply our construction to embedded alternating renewal process after some modification.

\paragraph{{acknowledgements}}
The author is grateful to L.~G.~Afanasyeva and \linebreak A.~Yu.~Ve\-re\-te\-n\-ni\-kov for detailed discussions and comments,  to H.~Thorisson for useful advices and remarks, and to M.~K.~Turtsynsky for the great help.

\end{document}